\documentclass[a4paper,11pt,twoside,reqno]{amsart}

\usepackage[utf8]{inputenc}
\usepackage[plainpages=false,pdfpagelabels=true]{hyperref}
\usepackage{amssymb,amsthm}
\usepackage{amsmath}
\usepackage[margin=1in]{geometry}
\usepackage{dutchcal}
\usepackage{xcolor}

\newtheorem{Theorem}{Theorem}[section]
\newtheorem{Prop}[Theorem]{Proposition}

\newtheorem{Thm}[Theorem]{Theorem}
\newtheorem{Cor}[Theorem]{Corollary}
\theoremstyle{definition}

\newtheorem{Bem}[Theorem]{Remark}
\newtheorem{Bsp}[Theorem]{Example}
\newcommand{\tr}{\operatorname{Tr}}
\newcommand{\Ric}{\operatorname{Ric}}
\newcommand{\Id}{\operatorname{Id}}
\newcommand{\Scal}{\operatorname{Scal}}
\newcommand{\Hess}{\operatorname{Hess}}

\newcommand{\Div}{\operatorname{div}}
\newcommand{\Int}{\operatorname{Int}}
\newcommand{\Ker}{\operatorname{Ker}}

\newcommand{\vol}{{\operatorname{vol}}}

\newcommand{\s}{\mathbb{S}}

\allowdisplaybreaks[1]

\renewcommand{\epsilon}{\varepsilon}

\numberwithin{equation}{section}

\newtheorem{innercustomgeneric}{\customgenericname}
\providecommand{\customgenericname}{}
\newcommand{\newcustomtheorem}[2]{%
  \newenvironment{#1}[1]
  {%
   \renewcommand\customgenericname{#2}%
   \renewcommand\theinnercustomgeneric{##1}%
   \innercustomgeneric
  }
  {\endinnercustomgeneric}
}

\newcustomtheorem{customthm}{Theorem}

\newcustomtheorem{customprop}{Proposition}

\title{On conformal-biharmonic maps and hypersurfaces}

\author{Volker Branding}
\address{Institute of Mathematics\\
University of Rostock, Ulmenstraße 69, 18057 Rostock, Germany}
\email{volker.branding@uni-rostock.de}

\author{Simona Nistor}
\address{Faculty of Mathematics\\
Al.I. Cuza University of Iasi, Blvd. Carol I, 11, 700506 Iasi, Romania}
\email{nistor.simona@ymail.com}

\author{Cezar Oniciuc}
\address{Faculty of Mathematics\\
Al.I. Cuza University of Iasi, Blvd. Carol I, 11, 700506 Iasi, Romania}
\email{oniciucc@uaic.ro}

\thanks{The first author gratefully acknowledges the support of the Austrian Science Fund (FWF) through the project ``Geometric Analysis of Biwave Maps'' (P34853).
}

\date{\today}


\begin{document}

\begin{abstract}
In this article we initiate a thorough geometric study of the conformal-bienergy functional which consists 
of the standard bienergy augmented by two additional curvature terms. The conformal-bienergy is
conformally invariant in dimension four and its precise structure is motivated by the Paneitz operator
from conformal geometry. The critical points of the conformal-bienergy are called conformal-biharmonic maps.

Besides establishing a number of basic results on conformal-biharmonic maps,
we pay special attention to conformal-biharmonic hypersurfaces in space forms.
For hypersurfaces in spheres, we determine all conformal-biharmonic hyperspheres and then we classify all conformal-biharmonic generalized Clifford tori.
Moreover, in sharp contrast to biharmonic hypersurfaces, we show that there also exist conformal-biharmonic hypersurfaces in hyperbolic spaces, pointing out a fundamental difference between biharmonic and conformal-biharmonic hypersurfaces.

Finally, we study the stability of conformal-biharmonic maps in several contexts. In particular, we analyze the stability of conformal-biharmonic hyperspheres in spheres and explicitly compute their index and nullity. For the equator $\mathbb{S}^m$ of the unit Euclidean sphere $\mathbb{S}^{m+1}$, which is trivially stable as a biharmonic map, we show that its conformal-biharmonic index is $0$, i.e., it is stable, for $m\leq 4$, whereas for $m\geq 5$ the conformal-biharmonic index equals $m+2$.
\end{abstract}

\maketitle

\section{Introduction}
The central object in the geometric calculus of variations is an energy functional whose precise form is most often fixed by the invariance under various symmetry operations. For maps between Riemannian manifolds, one of the most prominent geometric variational problems is the one corresponding to harmonic maps. In order to define harmonic maps, we consider two given Riemannian manifolds \((M^m,g)\) and \((N^n,h)\), with $M$ closed. Then, on the Fr\'echet manifold of all smooth maps \(\phi\colon M\to N\), we define the \emph{energy functional} by
\begin{align}
\label{energy}
E(\phi):=\frac{1}{2}\int_M|d\phi|^2 \ v_g.
\end{align}
The critical points of \eqref{energy} are characterized by the vanishing of the so-called \emph{tension field}
which is defined by
\begin{align*}
\tau(\phi):=\tr\bar\nabla d\phi, \qquad \tau(\phi)\in C\left(\phi^{-1}TN\right),
\end{align*}
where \(\bar\nabla\) represents the connection on the pull-back bundle \(\phi^{-1}TN\).
Solutions of \(\tau(\phi)=0\) are called \emph{harmonic maps}. The harmonic map equation is a semilinear elliptic partial differential equation of second order and many results concerning the existence and the qualitative properties of harmonic maps have been established over the years. 
We just want to mention the famous result of Eells and Sampson
which ensures the existence of harmonic maps from closed Riemannian manifolds to closed Riemannian manifolds of non-positive sectional curvature (see \cite{MR164306}). In order to prove this result, Eells and Sampson made use of the heat-flow method.
For the current status of research on harmonic maps we refer to the survey \cite{MR2389639} and the book \cite{MR2044031}.

A higher order generalization of harmonic maps that has received much attention in recent years is given by biharmonic maps. These maps also arise as the solutions of a geometric variational problem. Here, one considers the \emph{bienergy functional} for maps between Riemannian manifolds, which is given by
\begin{align}
\label{bienergy}
E_2(\phi):=\frac{1}{2}\int_M|\tau(\phi)|^2 \ v_g,
\end{align}
where $M$ is closed. The critical points of \eqref{bienergy} are called \emph{biharmonic maps} and are characterized by the vanishing of the \emph{bitension field} \(\tau_2(\phi)\), as it was proved by Jiang \cite{MR886529}
\begin{align}
\label{bitension}
0=\tau_2(\phi):=-\bar\Delta\tau(\phi)-\tr R^N(d\phi(\cdot), \tau(\phi))d\phi(\cdot).
\end{align}
The biharmonic map equation \(\tau_2(\phi)=0\) constitutes a semilinear elliptic partial differential equation of fourth order and the larger number of derivatives leads to additional technical difficulties in its analysis.
A direct inspection of the biharmonic map equation reveals that every harmonic map is automatically a solution of the biharmonic map equation. Hence, in the analysis of biharmonic maps, one is very much interested
in constructing biharmonic maps that are non-harmonic, which are called \emph{proper biharmonic}. In the case that \(M\) is closed and \(N\) has non-positive sectional curvature, the maximum principle implies that every biharmonic map is necessarily harmonic (see \cite{MR886529}) and, due to this reason, most of the research on biharmonic maps considers the case of a target manifold with positive curvature. If one tries to prove an existence result for biharmonic maps via the heat-flow method, one faces the difficulty that the bienergy \eqref{bienergy} is not coercive. One way to circumvent this problem is to assume that the target manifold has non-positive sectional curvature. Under this assumption and for \(\dim M\leq 4\), Lamm successfully studied the biharmonic map heat flow in \cite{MR2124627}, but due to the curvature assumption the limiting map will be harmonic. We also note that when $M$ is closed, any harmonic map is biharmonic stable, i.e, it is stable with respect to $E_2$. For the current status of research on biharmonic maps we refer to the book \cite{MR4265170} and for biharmonic hypersurfaces one may consult \cite{MR4410183}.

Comparing the symmetries of the energy functional \eqref{energy} with those of the bienergy functional \eqref{bienergy},
one can easily check that both functionals are invariant under isometries of the domain. However, regarding the invariance under conformal transformations, it turns out that the energy \eqref{energy} is invariant under conformal transformations of the domain in dimension two, but the bienergy \eqref{bienergy} does not enjoy conformal invariance in any dimension. Hence, from the point of view of conformal geometry, one can argue that the bienergy \eqref{bienergy} is not the most appropriate generalization of the energy \eqref{energy} as one of the fundamental symmetries is lost.

Fortunately, this drawback can be remedied by augmenting the bienergy \eqref{bienergy} with two additional lower order terms as follows (see \cite{MR2722777})
\begin{align}
\label{bienergy-conformal}
E^c_2(\phi):=\frac{1}{2}\int_M \left(|\tau(\phi)|^2+\frac{2}{3}\Scal |d\phi|^2-2\tr\langle d\phi\left(\Ric(\cdot)\right), d\phi(\cdot)\rangle\right) \ v_g.
\end{align}
We call \eqref{bienergy-conformal} the \emph{conformal-bienergy functional}.
The precise structure of \eqref{bienergy-conformal} is inspired by the Paneitz operator $P_g$ from conformal geometry, which is a fourth order elliptic operator acting on functions and which behaves well under conformal transformations of the domain when $\dim M =4$. The Paneitz equation $P_g(f)=0$ is the Euler-Lagrange equation of a scalar energy functional and moreover, when the target manifold is one-dimensional, the conformal-bienergy \eqref{bienergy-conformal} reduces exactly to this Paneitz energy. For more details on the Paneitz operator we refer to \cite[Chapter 4]{MR2521913}, see also \cite{MR2393291} for the original manuscript of Paneitz. For the sake of completeness, we will show by an explicit calculation that, if \(\dim M=4\), the conformal-bienergy \eqref{bienergy-conformal} is invariant under conformal transformations (see \cite{MR2722777}). Hence, one can argue that, from the point of view of conformal geometry,
the conformal-bienergy \eqref{bienergy-conformal} is the appropriate higher order generalization of the classical energy functional \eqref{energy}.

The critical points of the conformal-bienergy \eqref{bienergy-conformal} are those who satisfy (see \cite{MR2722777})
\begin{align}
	\label{eq-c-biharmonic-intro}
	0=\tau_2^c(\phi) & := \tau_2(\phi)-\frac{2}{3}\Scal\tau(\phi)+2\tr(\bar \nabla d\phi)(\Ric(\cdot),\cdot)+\frac{1}{3}d\phi(\nabla\Scal).
\end{align}
Solutions of \eqref{eq-c-biharmonic-intro} are called \emph{conformal-biharmonic maps}, or simply 
\emph{c-biharmonic maps}, and this article will provide the first systematic geometric investigation of such kind of maps.

The conformal-bienergy functional and the corresponding Euler-Lagrange equation were first addressed in the PhD thesis of B\'{e}rard \cite{MR2722777} for $\dim M=4$, 
see also the articles \cite{MR2449641} and \cite{MR3028564} originating from that thesis. 
B\'{e}rard used the terminology \emph{conformal harmonic maps} for the critical points of \eqref{bienergy-conformal} but, since they extend the (standard) biharmonic maps rather than harmonic maps, the terminology conformal-biharmonic map seems more appropriate even if $E_2^c$ is conformally invariant only when $\dim M=4$.

Besides the behavior of conformal-biharmonic maps under conformal transformations, there is a second difference compared to biharmonic maps: a harmonic map is not automatically a solution of the c-biharmonic map equation and thus it is interesting to see when a harmonic map is c-biharmonic and conversely. Another major deference is that a harmonic map, which is trivially biharmonic stable, is not automatically stable with respect to $E_2^c$, when it happens to be also c-biharmonic, as we will see in this paper.

There is a number of articles that already investigated some particular aspects of the conformal-bienergy functional. In the following we give a short overview of some of them. The conformal-bienergy appears in the bubbling analysis of (standard) biharmonic maps: if one considers a sequence of biharmonic maps in four-dimensional spaces, then quasi-biharmonic spheres separate
and these are a special case of conformal-biharmonic maps, for more details see the articles of Lamm \cite{MR2124627} and Wang \cite{MR2094320}. By making an assumption on both the Yamabe-invariant and the Q-curvature of the domain and assuming non-positive sectional curvature on the target, Biquard and Madani were able to extend the heat flow approach of Lamm (see \cite{MR2124627}) to the case of c-biharmonic maps from a four-dimensional domain (see \cite{MR2996776}). Gastel and Nerf studied an extrinsic version of the conformal-bienergy for maps into spheres in \cite{MR3070553}. Moreover, Lin and Zhu recently studied the uniqueness of c-biharmonic maps on locally conformally flat four-dimensional domains (see \cite{MR4847311}).

Let us give an overview of the main results obtained in this article.
After reviewing the basic aspects of the conformal-bienergy, we focus on c-biharmonic hypersurfaces in space forms of arbitrary dimension. First, we establish some general properties of c-biharmonic hypersurfaces in space forms. In particular, we prove a link between c-biharmonic hypersurfaces and Willmore hypersurfaces (Theorem \ref{thm-Willmore}). Using this, we observe that one important difference with respect to the biharmonic case is that there exist examples of c-biharmonic hypersurfaces with $4$ constant distinct principal curvatures in Euclidean spheres, including both minimal and non-minimal cases (Remark \ref{remark-4-curvatures}), whereas all known examples of proper biharmonic hypersurfaces have $1$ or $2$ constant distinct principal curvatures. Moreover, a rigidity result for c-biharmonic Einstein hypersurfaces is obtained (Theorem \ref{c-biharmonic-umbilical}), revealing the importance of $\dim M=4$. Then, we explicitly classify the c-biharmonic hypersurfaces in Euclidean spaces, Euclidean spheres and hyperbolic spaces, with $1$ or $2$ constant distinct principal curvatures. 


In contrast to proper biharmonic small hyperspheres in a unit Euclidean sphere $\mathbb{S}^{m+1}$, c-biharmonic small hyperspheres $\mathbb{S}^{m}(r)$ of radius $r$ exist only for $m\leq 4$, as is demonstrated by the following 

\begin{customthm}{\ref{thm:inclusion-sphere}}
The hypersphere $\s^m(r)$ is c-biharmonic in $\s^{m+1}$ if and only if
	\begin{enumerate}
		\item [i)] $r=1$, i.e., $\s^m(r)$ is totally geodesic in $\s^{m+1}$ and for any $m\geq 1$;
		\item [ii)] $m=1$ or $m=3$ and $r=1/\sqrt{2}$;
		\item [iii)] $m=2$ and $r=1/\sqrt{3}$;
		\item [iv)] $m=4$ and $r=\sqrt{3}/2$.
	\end{enumerate}
\end{customthm}

In the case of c-biharmonic generalized Clifford tori, we obtain their complete classification which shows the peculiarity of $m=4$.

\begin{customthm}{\ref{thm:c-biharmonic-clifford-complete}}
	The generalized Clifford torus $\s^{m_1}\left(r_1\right)\times \s^{m_2}\left(r_2\right)$, where  $m_1+m_2=m$ and $r_1^2+r_2^2=1$, is c-biharmonic in $\s^{m+1}$ if and only if one of the following cases holds
	\begin{itemize}
		\item[i)] $m_1=m_2=2$ and either 
		$$
		r_1^2=r_2^2=\frac{1}{2}, \qquad\text{or}\qquad r_1^2=\frac{1}{2}\left(1-\frac{1}{\sqrt{3}}\right),\  r_2^2=\frac{1}{2}\left(1+\frac{1}{\sqrt{3}}\right).
		$$
		\item[ii)] $m_1\neq 2$ or $m_2\neq 2$ and 
		$$
		r_1^2=\frac{T_\ast}{1+T_\ast}, \ r_2^2=\frac{1}{1+T_\ast},
		$$
		where $T_\ast$ is the unique solution of the polynomial equation
		\begin{equation*}
			a_3T^3 + a_2T^2+a_1T+a_0=0,
		\end{equation*}
		with the coefficients given by
		\begin{equation*}
			\left\{
			\begin{array}{l}
				a_0=m_1\left(2m_1^2-5m_1+6\right)\\\\
				a_1=m_1\left(\left(m_1-m_2\right)^2+\left(m_1-\frac{11}{2}\right)^2+\left(m_2-3\right)^2-\frac{133}{4}\right)\\\\
				a_2=-m_2\left(\left(m_2-m_1\right)^2+\left(m_2-\frac{11}{2}\right)^2+\left(m_1-3\right)^2-\frac{133}{4}\right)\\\\
				a_3=-m_2\left(2m_2^2-5m_2+6\right)
			\end{array}
			\right..
		\end{equation*}	
	\end{itemize}
\end{customthm}

One of the major differences between biharmonic and c-biharmonic submanifolds is the fact that there exist
c-biharmonic hypersurfaces in hyperbolic space, while there are no known proper biharmonic hypersurfaces
in spaces of non-positive constant sectional curvature. 
This fact is illustrated by the following result which can be seen as a counterpart of Theorem \ref{thm:inclusion-sphere}.

\begin{customthm}{\ref{prop-c-biharmonic-equidistant}}
The hypersurface $M^m=\left\{\bar{x}\in\mathbb{H}^{m+1}\ | \ x^{1}=r\geq 0\right\}$ is c-biharmonic in $\mathbb{H}^{m+1}$ if and only if 
\begin{itemize}
	\item[i)]  $r=0$, i.e., $M^m$ is totally geodesic in $\mathbb{H}^{m+1}$;
	\item[ii)] $m\geq 5$ and
	$$
	r^2=\frac{2m^2-11m+6}{6m}.
	$$ 
\end{itemize} 
\end{customthm}

As we have an explicit geometric description of c-biharmonic hyperspheres, we also investigate their stability. In order to approach this question, we compute the Jacobi operator associated with c-biharmonic maps $\phi:\mathbb{S}^m(r)\to\mathbb{S}^{n}$, which is the following linear, elliptic differential operator of order four
\begin{align}
\label{eq:intro-jacobi}
J^c_2(V) = &\bar\Delta\bar\Delta V+\bar\Delta\left(\tr \langle V,d\phi(\cdot)\rangle d\phi(\cdot)-|d\phi|^2 V \right)+2\langle d\tau(\phi),d\phi\rangle V+|\tau(\phi)|^2V \nonumber\\
\nonumber& - 2\tr \langle V,d\tau(\phi)(\cdot)\rangle d\phi(\cdot)-2\tr\langle \tau(\phi),dV(\cdot)\rangle d\phi(\cdot)-\langle \tau(\phi),V\rangle \tau(\phi)+\tr\langle d\phi(\cdot),\bar{\Delta}V\rangle d\phi(\cdot)\\
\nonumber& + \tr\langle d\phi(\cdot),\left(\tr\langle V,d\phi(\cdot)\rangle d\phi(\cdot)\right)\rangle d\phi(\cdot)-2|d\phi|^2\tr\langle d\phi(\cdot),V\rangle d\phi(\cdot)+2\langle dV,d\phi\rangle \tau(\phi)\\
 & - |d\phi|^2\bar{\Delta}V+|d\phi|^4V
+\frac{2}{3}\frac{(m-1)(m-3)}{r^2}\left(\bar\Delta V+\tr\langle d\phi(\cdot),V\rangle d\phi(\cdot)-|d\phi|^2V\right),
\end{align}
where \(V\in C(\phi^{-1}T\mathbb{S}^{m+1})\).
For a closed domain, the spectrum of \eqref{eq:intro-jacobi} is discrete 
\begin{align*}
\lambda_1<\lambda_2<\cdots <\lambda_j<\cdots
\end{align*}
and tends to \(+\infty\) as \(j\to\infty\).

By \(Q_j\) we denote the eigenspace associated with the eigenvalue \(\lambda_j\).
Then, we define
\begin{align*}
\operatorname{Index}(\phi):=\sum_{\lambda_j<0}\dim (Q_j).
\end{align*}
In addition, we define the nullity of \(\phi\) as follows
\begin{align*}
\operatorname{Nullity}(\phi):=\dim\{V\in C(\phi^{-1}TN) \ | \ J^c_2(V)=0\}.
\end{align*}
We say that a c-biharmonic map is \emph{stable} if its index is zero, i.e., \(\operatorname{Index}(\phi)=0.\) We recall that any map from a closed manifold which happens to be both harmonic and c-biharmonic is trivially biharmonic stable but not necesarily c-biharmonic stable.


Now we are ready to present the index and nullity of the c-biharmonic hyperspheres
provided by Theorem \ref{thm:inclusion-sphere}. For the stability of c-biharmonic great hyperspheres we have

\begin{customthm}{\ref{thm:stability-equator}}
  We consider $\iota:\mathbb{S}^m\to\mathbb{S}^{m+1}$ the canonical inclusion of the totally geodesic hypersphere $\mathbb{S}^m$, thought of as a c-biharmonic map. Then, we have
  \begin{itemize}
	\item [i)] if $m\in\{1,2,3,4\}$, then the index of $\mathbb{S}^m$ is $0$, i.e., $\mathbb{S}^m$ is stable;
    \item[ii)] if $m\geq 5$, then the index of $\mathbb{S}^m$ is $m+2$.
  \end{itemize}
\end{customthm}

This is quite surprising as there is no stable closed proper biharmonic submanifold in a Euclidean sphere (see \cite{MR886529}). 

We also note that, according to a result of Xin (see \cite{MR0587168}), the stability with respect to the energy functional $E$ of a non-constant harmonic map from $\mathbb{S}^m$ in an arbitrary manifold $N^n$ forces $m$ to be less or equal to $2$. But a harmonic map from $\mathbb{S}^m$ is also c-biharmonic and, as we will see in Theorem \ref{xin-generalization}, its stability with respect to $E_2^c$ can occur only for $m$ less or equal to $4$. In particular, the identity map of $\mathbb{S}^m$, which is both armonic and c-harmonic, is stable with respect to $E$ for $m$ less or equal to $2$ (see \cite{MR0375386}), but is stable with respect to $E_2^c$ for $m$ less or equal to $4$ (Corollary \ref{stability-identity}). Therefore, there are situations when the conformal-bienergy functional $E^c_2$ has better stability properties than the classical energy $E$.

For the stability of c-biharmonic small hyperspheres we obtain

\begin{customthm}{\ref{thm:stability-hypersphere}}
We consider $\iota:\mathbb{S}^m(r)\to\mathbb{S}^{m+1}$ the canonical inclusion of the small hypersphere of radius $r$. We have
\begin{itemize}
		\item [i)] if $m=1$ or $m=3$ and $r=1/\sqrt{2}$, then the index of $\mathbb{S}^m(r)$ is $1$ and its nullity is $(m+1)(m+2)/2$;
		\item[ii)] if $m=2$ and $r=1/\sqrt{3}$, then the index of $\mathbb{S}^m(r)$ is $1$ and its nullity is $6$;
		\item[iii)] if $m=4$ and $r=\sqrt{3}/2$, then the index of $\mathbb{S}^m(r)$ is $1$ and its nullity is $20$.
\end{itemize}
\end{customthm}

Our results clearly emphasize that c-biharmonic maps have a rich structure
and, although the c-bienergy only differs from the bienergy by lower order terms, the results show that biharmonic and conformal-biharmonic maps are the solutions of two very different variational problems. We believe that c-biharmonic maps are an interesting alternative to the biharmonic maps. 

Throughout this article we will use the following sign conventions: 
for the Riemannian curvature tensor field we use 
$$
R(X,Y)Z=\left[\nabla_X,\nabla_Y\right]Z-\nabla_{[X,Y]}Z,
$$ 
for the Ricci tensor field 
$$
g(\Ric(X),Y)=\Ric(X,Y)=\tr \left\{Z\to R(Z,X)Y\right\},
$$
and the scalar curvature is given by
$$
\Scal=\tr\Ric.
$$
The trace is taken with respect to the domain metric and we write $\tr$ instead of $\tr_g$. In general, for the connection $\nabla^\phi$ on the pull-back bundle $\phi^{-1} TN$ we use the symbol $\bar{\nabla}$.
For the rough Laplacian on the pull-back bundle $\phi^{-1} TN$ we employ the geometers sign convention
$$
\bar\Delta = -\tr\left(\bar\nabla\bar\nabla-\bar\nabla_\nabla\right).
$$
In particular, this implies that the Laplace operator has a non-negative spectrum. For the metric $g$ on $TM$ and also for the metric on the other vector bundles involved, we use the same symbol $\langle\cdot,\cdot\rangle$. All manifolds are assumed to be smooth, connected and without boundary. Thus, in our paper, all compact manifolds are closed.  

\section{Preliminaries on conformal-biharmonic maps}

\subsection{First variation formula}
Let $\left(M^m,g\right)$ and $\left(N^n,h\right)$ be two fixed Riemannian manifolds and assume, for simplicity, that $M$ is compact. On the Fr\'echet manifold of all smooth maps between $M$ and $N$ we define the following energy functional
\begin{align*}
E_2^c(\phi) & = \frac{1}{2}\int_M \left(|\tau(\phi)|^2+\frac{2}{3}\Scal |d\phi|^2-2\tr\langle d\phi\left(\Ric(\cdot)\right), d\phi(\cdot)\rangle\right) \ v_g\\
& = E_2(\phi)+\int_M \left(\frac{1}{3}\Scal |d\phi|^2-\tr\langle d\phi\left(\Ric(\cdot)\right), d\phi(\cdot)\rangle\right) \ v_g. \nonumber 
\end{align*} 
We may say that $E_2^c$ is $E_2$ with a small correction. The additional terms of first order render $E_2^c$ conformally invariant in dimension $m=4$. 

We call $E_2^c$ the conformal-bienergy and its critical points conformal-biharmonic maps, for short c-biharmonic maps. Moreover, when $\phi:\left(M^m,g\right)\to\left(N^n,h\right)$ is c-biharmonic and also an isometric immersion, then we say that $M$ is a \emph{c-biharmonic submanifold} of $N$.

\begin{Prop}\label{prop-energies-equal}
The conformal-bienergy coincides with the bienergy functional in each of the following cases:
\begin{itemize}
	\item [i)]  if the dimension of the domain manifold is $m=1$;
	\item [ii)] if $m=3$ and $M^3$ is an Einstein manifold;
	\item[iii)] if the domain manifold $M^m$ is Ricci-flat.
\end{itemize}
\end{Prop}
	
Even if $E_2^c$ and $E_2$ seem to be similar at first glance, they have in fact a very different behavior as is illustrated by the following example. Let $\phi_t:\mathbb{S}^{m}\to\mathbb{S}^{m+1}$ be a smooth map defined by
$$
\phi_t\left(\overline{x}\right)=\left(\sqrt{1-t^2}\ \overline{x},t\right), \qquad \forall \overline{x}\in\mathbb{S}^m, \ t\in (-1,1).
$$
The map $\phi_0=\phi$ is the equator of the unit Euclidean sphere $\mathbb{S}^{m+1}$, so it is harmonic. Clearly, $\phi$ is a critical point of $E_2$ and one can see that it is also a critical point of $E_2^c$. While 
$$
E_2\left(\phi_t\right)>E_2(\phi)=0, \qquad \forall t\neq 0,
$$ 
in the case of the conformal-bienergy, by standard computations we obtain that
$$
E_2^c\left(\phi_t\right)>E_2^c(\phi), \qquad \forall t\neq 0,\  m\in \left\{1,2,3,4\right\}
$$
and
$$
E_2^c\left(\phi_t\right)<E_2^c(\phi), \qquad \forall t\neq 0,\  m\geq 5.
$$
If we set 
$$
h_m(t)=E_2\left(\phi_t\right)=\frac{1}{2}m^2\omega_m\left(1-t^2\right)t^2, \qquad \forall t\in (-1,1),
$$ 
where $\omega_m=\vol\left(\mathbb{S}^m\right)$, one can obtain that $t=0$ is a minimal point of $h_m$, and $t=\pm 1/\sqrt{2}$ are maximal points of $h_m$. Clearly, $\phi_{\pm 1/\sqrt{2}}$ represent the biharmonic (and non-minimal) small hyperspheres of $\mathbb{S}^{m+1}$. 

Now, if we set
$$
h_m^c(t)=E_2^c\left(\phi_t\right)=\frac{1}{2}m^2\omega_m\left(1-t^2\right)\left(t^2+\frac{2}{3}\frac{(m-1)(m-3)}{m}\right), \qquad \forall t\in (-1,1),
$$
we can see that $t=0$ is a minimal point of $h_m$, when $m\in\left\{1,2,3,4\right\}$, and a maximal point of $h_m$, when $m\geq 5$. The other non-zero solutions $t_{\ast}$ of 
$$
\frac{d}{dt}\left\{h_m^c(t)\right\}=0
$$ 
are given by $t_{\ast}^2=\left(-2m^2+11m-6\right)/(6m)$, when $m\in\left\{1,2,3,4\right\}$, and they are maximal points of $h_m$. The maps $\phi_{t_{\ast}}$ represent the c-biharmonic (and non-minimal) small hyperspheres of $\mathbb{S}^{m+1}$ (see Theorem \ref{thm:inclusion-sphere}). 

\begin{Bem}
When $M^m$ is an Einstein manifold, $m\neq 3$, with $\Ric=\lambda \Id$, where $\lambda$ is a real constant, then 
$$
E_2^c(\phi)=\frac{1}{2}\int_M \left(|\tau(\phi)|^2+\frac{2}{3}(m-3)\lambda\left|d\phi\right|^2\right) \ v_g,
$$
which is similar to the energy functional studied in \cite{MR4058514} and \cite{MR4142862}.
\end{Bem}

In the following, for the sake of completness, we derive the first variation formula of the conformal-bienergy \eqref{bienergy-conformal} (see also \cite{MR4847311} for a more analytical approuch).

\begin{Theorem}[\cite{MR2722777}]
The critical points of the conformal-bienergy are given by
\begin{align*}
0=\tau_2^c(\phi) =\tau_2(\phi)-\frac{2}{3}\Scal\tau(\phi)+2\tr(\bar \nabla d\phi)(\Ric(\cdot),\cdot)+\frac{1}{3}d\phi(\nabla\Scal).
\end{align*}
\end{Theorem}
\begin{proof}

Let $\left\{\phi_t\right\}_{t\in\mathbb{R}}$ be a smooth variation of $\phi$, i.e., a smooth map $\Phi:\mathbb{R}\times M\to N$, $\Phi(t,p)=\phi_t(p)=\phi_p(t)$, for any $(t,p)\in\mathbb{R}\times M$, satisfying
$$
\Phi(0,p)=\phi_0(p)=\phi_p(0)=\phi(p), \qquad \forall p\in M. 
$$
The variational vector field $V\in C\left(\phi^{-1}TN\right)$ associated with the variation $\left\{\phi_t\right\}_{t\in\mathbb{R}}$ of $\phi$ is defined by
$$
V(p)=\left.\frac{d}{dt}\right|_{t=0}\left\{\phi_p(t)\right\}=d\Phi_{(0,p)}\left(\frac{\partial}{\partial t}\right)\in T_{\phi(p)}N, \qquad \forall p\in M.
$$
It is known that 
\begin{align}\label{eq1}
\left.\frac{d}{dt}\right|_{t=0}\left\{E_2\left(\phi_t\right)\right\}=\int_M\langle\tau_2(\phi), V\rangle \ v_g.
\end{align}
In order to compute 
$$
\left.\frac{d}{dt}\right|_{t=0}\left\{\Scal \left|d\phi_t\right|^2\right\},
$$
we fix an arbitrary point $p\in M$ and consider a geodesic frame field $\left\{X_i\right\}_{i=\overline{1,m}}$ around $p$. We have
\begin{align*}
\left.\frac{d}{dt}\right|_{t=0}\left\{\sum_{i=1}^{m}\Scal(p)\langle d\phi_t\left(X_i\right),d\phi_t\left(X_i\right)\rangle \right\} & = \left(\frac{\partial}{\partial t}\right)_{(0,p)}\left\{\sum_{i=1}^{m}\Scal(p)\langle d\Phi\left(X_i\right),d\Phi\left(X_i\right)\rangle\right\} \\
 & = \sum_{i=1}^{m}2\Scal(p)\langle \nabla^\Phi_{X_i}d\Phi\left(\frac{\partial}{\partial t}\right),d\Phi\left(X_i\right)\rangle_{(0,p)}	\\
 & = \sum_{i=1}^{m} 2\Scal(p)\langle \bar\nabla_{X_i}V, d\phi\left(X_i\right)\rangle_p  \\
 & = \sum_{i=1}^{m} 2\Scal(p)\left\{X_i\langle V,d\phi\left(X_i\right)\rangle-\langle V, \bar\nabla_{X_i}d\phi\left(X_i\right)\rangle\right\}_p.
\end{align*}
We set
$$
Y_1=\Scal \tr \langle V, d\phi\left(\cdot\right)\rangle (\cdot)\in C(TM).
$$
It is easy to see that, at $p$, 
$$
\Div Y_1=\langle V, d\phi(\nabla\Scal)\rangle+\Scal\sum_{i=1}^m X_i\langle V, d\phi\left(X_i\right)\rangle,
$$
so, at p, we have 
\begin{equation}\label{eq-scal}
\left.\frac{d}{dt}\right|_{t=0}\left\{\Scal \left|d\phi_t\right|^2\right\}=2\left(\Div Y_1-\langle V, d\phi\left(\nabla \Scal \right)+\Scal\tau(\phi)\rangle\right).
\end{equation}
Now, we want to compute at the same point $p\in M$, 
$$
\left.\frac{d}{dt}\right|_{t=0}\left\{\tr\langle d\phi\left(\Ric(\cdot)\right),d\phi(\cdot)\rangle \right\}.
$$
A direct computation shows that
\begin{align*}
\left.\frac{d}{dt}\right|_{t=0}\left\{\sum_{i=1}^m\langle d\phi_t\left(\Ric\left(X_i\right)\right),d\phi_t\left(X_i\right)\rangle_p \right\} & = \left(\frac{\partial}{\partial t}\right)_{(0,p)}\left\{\sum_{i=1}^m\langle d\Phi\left(\Ric\left(X_i\right)\right),d\Phi\left(X_i\right)\rangle\right\} \\
& =\sum_{i=1}^m \left\{\Ric\left(X_i\right)\langle V, d\phi\left(X_i\right)\rangle+X_i\langle d\phi\left(\Ric\left(X_i\right)\right), V\rangle \right.\\
& \ \left. -\langle V,2\left(\bar \nabla d\phi\right)\left(\Ric\left(X_i\right),X_i\right)+d\phi\left(\left(\nabla \Ric\right)\left(X_i,X_i\right)\right)\rangle\right\}_p.
\end{align*}
We set
$$
Y_2=\tr\langle V, d\phi(\cdot)\rangle \Ric (\cdot)\in C(TM)
$$
and 
$$
Y_3=\tr\langle d\phi\left(\Ric(\cdot)\right), V \rangle(\cdot)\in C(TM).
$$
By a straightforward computation at $p$ we get
$$
\Div Y_2=\sum_{i=1}^m \Ric\left(X_i\right)\langle V, d\phi\left(X_i\right)\rangle + \langle V, d\phi\left(\tr\left(\nabla \Ric\right)(\cdot,\cdot)\right)\rangle
$$
and 
$$
\Div Y_3=\sum_{i=1}^m X_i\langle d\phi\left(\Ric\left(X_i\right)\right), V\rangle.
$$
Consequently, at $p$ we obtain
\begin{equation}\label{eq3}
\left.\frac{d}{dt}\right|_{t=0}\left\{\tr\langle d\phi\left(\Ric(\cdot)\right),d\phi(\cdot)\rangle \right\}=\Div Y_2+\Div Y_3-2\langle V, d\phi\left(\tr\left(\nabla\Ric\right)(\cdot,\cdot)\right)+\tr \left(\bar \nabla d\phi\right)\left(\Ric(\cdot),\cdot\right)\rangle
\end{equation}
Using \eqref{eq1}, \eqref{eq-scal}, \eqref{eq3}, we achieve
\begin{align*}
\left.\frac{d}{dt}\right|_{t=0}\left\{E_2^c\left(\phi_t\right)\right\}  = \int_M\langle &\tau_2(\phi)-\frac{2}{3}\Scal\tau(\phi)+2\tr \left(\bar \nabla d\phi\right)\left(\Ric(\cdot),\cdot\right) \\
&+ 2 d\phi\left(\tr\left(\nabla \Ric\right)(\cdot,\cdot)\right)-\frac{2}{3}d\phi\left(\nabla\Scal\right),V\rangle \ v_g.
\end{align*}
We conclude that $\phi$ is a critical point of $E_2^c(\phi)$ if and only if $\tau_2^c (\phi)=0$, where
\begin{equation*}
\tau_2^c (\phi)=\tau_2(\phi)-\frac{2}{3}\Scal\tau(\phi)+2\tr \left(\bar \nabla d\phi\right)\left(\Ric(\cdot),\cdot\right)+ 2 d\phi\left(\tr\left(\nabla \Ric\right)(\cdot,\cdot)\right)-\frac{2}{3}d\phi\left(\nabla\Scal\right).
\end{equation*}
Moreover, since
$$
\tr(\nabla\Ric)(\cdot,\cdot)=\Div \Ric =\frac{1}{2}\nabla\Scal,
$$
we have that 
\begin{equation*}
\tau_2^c(\phi)=\tau_2(\phi)-\frac{2}{3}\Scal\tau(\phi)+2\tr(\bar \nabla d\phi)(\Ric(\cdot),\cdot)+\frac{1}{3}d\phi(\nabla\Scal)
\end{equation*}
completing the proof.
\end{proof}

\begin{Bem}
When $M^m$ is non-compact, a smooth map $\phi:\left(M^m,g\right)\to \left(N^n,h\right)$ is said to be c-biharmonic if
$$
\left.\frac{d}{dt}\right|_{t=0}\left\{E_2^c\left(\phi_t;D\right)\right\}=0
$$
for all compact regular domains $D$ and all smooth variations $\left\{\phi_t\right\}_{t\in\mathbb{R}}$ of $\phi$ supported in $D$, i.e., $\phi_t=\phi$ on $M\setminus\Int(D)$, for any $t\in\mathbb{R}$. Using the divergence theorem and the fact that $V$ and $\bar{\nabla}V$ vanish on the boundary $\partial D$ of $D$, by long but standard computations we obtain that $\phi$ is c-biharmonic if and only if $\tau_2^c(\phi)=0$.
\end{Bem}

\begin{Bem}
Up to a multiplicative real constant, $\tau_2^c(\phi)$ is invariant under homothetic transformations of the domain and codomain manifolds. 
\end{Bem}

\begin{Bem}
When $M^m$ is non-compact, $\tau_2^c(\phi)=\tau_2(\phi)$, in each of the cases mentioned in Proposition \ref{prop-energies-equal}.
\end{Bem}

\begin{Bem}
In the particular case when $\phi:M^4\to\mathbb{R}$, then $\tau_2^c$ is the Paneitz operator.
\end{Bem}

Since the identity map of a Riemannian manifold is totally geodesic, it is not difficult to check the following result.

\begin{Prop} \label{prop:identity}
The identity map $\Id:\left(M^m,g\right)\to\left(M^m,g\right)$ is c-biharmonic if and only if $\Scal$ is constant.
\end{Prop}

As we have already mentioned in the Introduction, a harmonic map is not always c-biharmonic, contrary to the biharmonic case. However, the following result presents a simple instance when a harmonic map is  also c-biharmonic.

\begin{Prop}[\cite{MR2722777}]
Let	$\phi:M^m\to N^n$ be a harmonic map, where $M$ is an Einstein manifold with $\Ric=\lambda\Id$ and $\lambda$ is a real constant. Then, $\phi$ is c-biharmonic.
\end{Prop}

It turns out that c-biharmonic maps are sensitive to the curvatures of both domain and target manifold.
Nevertheless, compared to biharmonic maps, the situation appears to be less constrained. For instance, when the target manifold has non-positive curvature and the domain is non-compact, several examples of c-biharmonic submanifolds can be found, while the number of biharmonic ones remains quite limited. The following result presents a particular situation where the role of curvatures is involved. Denoting by $\mathbb{S}^m(r)$ the Euclidean sphere of radius $r$ and by $N^n(c)$ a space form of constant sectional curvature $c$, we have the following result.

\begin{Prop}\label{prop:restriction-of-radius}
Let $\phi:\mathbb{S}^m(r)\to N^n(c)$ be a c-biharmonic submanifold. Then, it is either minimal, or
$$
c>\frac{2}{3}\frac{(m-1)(m-3)}{m}\frac{1}{r^2}.
$$
In particular, when $c$ is non-positive, $m$ has to be equal to $2$.
\end{Prop}   

\begin{proof}
By straightforward computation we obtain
\begin{align*}
\tau_2^c(\phi) & =-\bar{\Delta} \tau(\phi)+\left(mc-\frac{2}{3}(m-1)(m-3)\frac{1}{r^2}\right)\tau(\phi)\\
& =0.
\end{align*}
Taking the inner product with $\tau(\phi)$ and then integrating, we conclude.
\end{proof}

\begin{Bem}
If in the above result we consider an open part of $\mathbb{S}^m(r)$ instead of the whole $\mathbb{S}^m(r)$ but we assume $\left|\tau(\phi)\right|$ is constant, the same conclusion holds.    
\end{Bem}

\begin{Prop}
Let $\phi:M^m\to N^n$ be a c-biharmonic map, where $M$ is a compact Einstein manifold with $\Ric=\lambda \Id$ and $\lambda$ is a real constant, and suppose that $N$ has non-positive sectional curvature. If $\left(m-3\right)\lambda\geq 0$, then $\phi$ is harmonic.
\end{Prop}

\begin{proof}
We compute
\begin{align*}
\frac{1}{2}\Delta|\tau(\phi)|^2=&-\left|\bar\nabla\tau(\phi)\right|^2
+\langle\bar\Delta\tau(\phi),\tau(\phi)\rangle \\
=& -\left|\bar\nabla\tau(\phi)\right|^2
+\tr\langle R^N(d\phi(\cdot),\tau(\phi))\tau(\phi),d\phi(\cdot)\rangle
-\frac{2}{3}\left(m-3\right)\lambda|\tau(\phi)|^2.
\end{align*}
By applying the maximum principle we obtain that $\bar\nabla\tau(\phi)=0$. Using
\begin{align*}
\int_M \langle \bar\nabla\tau(\phi), d\phi\rangle \ v_g &= -\int_M \left|\tau(\phi)\right|^2 \ v_g \\
& = 0
\end{align*}
we obtain the claim.
\end{proof}

When the dimension $m=4$ and $\lambda\geq 0$, in particular when $M=\mathbb{S}^4(r)$, we can recover a result of \cite{MR2722777} (see also Lamm \cite[Theorem 2.1]{MR2124627}), i.e.,

\begin{Cor}[\cite{MR2722777}]
Let $\phi:M^4\to N^n$ be a smooth map, where $M$ is a compact Einstein manifold with $\Ric=\lambda \Id$ and $\lambda$ is a real constant, and suppose that $N$ has non-positive sectional curvature. If $\lambda\geq 0$, then $\phi$ is c-biharmonic if and only if it is harmonic.
\end{Cor}
\subsection{Conformal invariance}
We have the following known result.

\begin{Thm}[\cite{MR2722777}]
Let $\phi:\left(M^m,g=\langle\cdot,\cdot\rangle\right)\to\left(N^n,h\right)$ be a smooth map and \(\rho\in C^\infty(M)\). 
In addition, consider $\tilde{\phi}:\left(M^m,\tilde{g}=e^{2\rho}g\right)\to \left(N^n,h\right)$ defined by $\tilde{\phi}(p)=\phi(p)$, for any $p\in M$. Assume that $M$ is compact and $m=4$. Then,
$$
E_2^c\left(\tilde{\phi}\right)=E_2^c(\phi) \qquad \text{and} \qquad \tau_2^c\left(\tilde{\phi}\right)=e^{-4\rho}\tau_2^c(\phi).
$$
\end{Thm}

For the sake of completeness, we present a short proof of this result.

\begin{proof}
	
We use the symbol $\tilde{}$ for the objects corresponding to $\tilde{g}$ and we recall the following formulas.

The Levi-Civita connection associated with $\tilde{g}$ is
$$
\tilde{\nabla}_X Y=\nabla_X Y+\left(X\rho\right)Y+\left(Y\rho\right)X-\langle X,Y\rangle \nabla \rho, \qquad \forall X,Y\in C(TM).
$$
The curvature tensor field $\tilde{R}$ is given by
\begin{align*}
\tilde{R}(X,Y)Z = & R(X,Y)Z-\left(\Hess\rho\right)(Y,Z)X+\left(\Hess\rho\right)(X,Z)Y+\left(Y\rho\right)\left(Z\rho\right)X-\left(X\rho\right)\left(Z\rho\right)Y\\
& -\left(\Hess\rho\right)\left(\langle Y,Z\rangle X-\langle X,Z\rangle Y\right) - \left|\nabla\rho\right|^2 \left(\langle Y,Z\rangle X-\langle X,Z\rangle Y\right)\\
& +\left(\left(X\rho\right)\langle Y,Z\rangle-\left(Y\rho\right)\langle X,Z\rangle\right) \nabla \rho.
\end{align*}
The Ricci curvature $\widetilde{\Ric}$ has the following expression
\begin{equation*}
\widetilde{\Ric}(X) = e^{-2\rho}\left(\Ric(X)-(m-2)\left(\left(\Hess\rho\right)(X)-\left(X\rho\right)\nabla\rho\right) +\left(\Delta\rho-(m-2)\left|\nabla \rho\right|^2\right)X\right).
\end{equation*} 
The formula for the scalar curvature $\widetilde{\Scal}$ is
\begin{equation*}
\widetilde{\Scal}=e^{-2\rho}\left(\Scal+(m-1)\left(2\Delta\rho-(m-2)\left|\nabla\rho\right|^2\right)\right).
\end{equation*} 
For more details on these formulas we refer to \cite[Theorem 1.159]{MR2371700}.

The volume form associated with $\tilde{g}$ is given by
$$
v_{\tilde{g}}=e^{m\rho}v_g.
$$
Next, we compute all terms which appear in the expression of $E_2^c\left(\tilde{\phi}\right)$ in terms of $\phi$ and $\langle\cdot,\cdot\rangle$.
We have 
$$
\tau\left(\tilde{\phi}\right)=e^{-2\rho}\left(\tau(\phi)+(m-2)d\phi\left(\nabla\rho\right)\right)
$$
and therefore
$$
\left|\tau\left(\tilde{\phi}\right)\right|^2=e^{-4\rho}\left(\left|\tau(\phi)\right|^2+(m-2)^2\left|d\phi\left(\nabla \rho\right)\right|^2+2(m-2)\langle \tau(\phi),d\phi \left(\nabla\rho\right)\rangle\right).
$$
The next terms can be written as 
$$
\widetilde{\Scal}\left|d\tilde{\phi}\right|^2 = e^{-4\rho}\left(\Scal+(m-1)\left(2\Delta\rho-(m-2)\left|\nabla \rho\right|^2\right)\right)\left|d\phi\right|^2
$$
and
\begin{align*}
\widetilde{\tr} \langle d\tilde{\phi}\left(\widetilde{\Ric}(\cdot)\right),d\tilde{\phi}(\cdot)\rangle  & = e^{-4\rho}\left( \tr\langle d\phi\left(\Ric(\cdot)\right), d\phi(\cdot)\rangle +\Delta\rho\left|d\phi\right|^2+\frac{1}{2}(m-2)\langle \nabla \rho,\nabla \left(|d\phi|^2\right)\rangle \right. \\
& \quad \left. - (m-2)\left(\Div Y_4-\langle \tau(\phi), d\phi(\nabla \rho)\rangle- \left|d\phi(\nabla\rho)\right|^2+\left|\nabla\rho\right|^2\left|d\phi\right|^2 \right)\vphantom{\frac 1 2}\right),
\end{align*}
where 
$$
Y_4=\tr\langle d\phi(\nabla \rho), d\phi(\cdot)\rangle(\cdot)\in C(TM).
$$
Therefore, we get
\begin{align*}
E_2^c\left(\tilde{\phi}\right) =\frac{1}{2}\int_M & e^{(m-4)\rho} \left(\left|\tau(\phi)\right|^2+\frac{2}{3}\Scal \left|d\phi\right|^2-2\tr\langle d\phi\left(\Ric(\cdot)\right), d\phi(\cdot)\rangle \right.\\
&\left.+(m-2)(m-4)\left|d\phi\left(\nabla \rho\right)\right|^2 -\frac{2}{3}(m-2)(m-4) \left|\nabla\rho\right|^2\left|d\phi\right|^2  \right. \\
&\left. +2(m-2)\Div Y_4 +\frac{2(2m-5)}{3}\Delta\rho\left|d\phi\right|^2 -(m-2)\langle \nabla \rho,\nabla \left(|d\phi|^2\right)\rangle  \right) \ v_g.
\end{align*}
In the following, if we consider $m=4$, we obtain
\begin{equation*}
E_2^c\left(\tilde{\phi}\right) =\frac{1}{2}\int_M \left(\left|\tau(\phi)\right|^2+\frac{2}{3}\Scal \left|d\phi\right|^2-2\tr\langle d\phi\left(\Ric(\cdot)\right), d\phi(\cdot)\rangle \right) \ v_g,
\end{equation*}
i.e., $E_2^c\left(\tilde{\phi}\right) =E_2^c\left(\phi\right)$.

Furthermore, considering $\left\{\phi_t\right\}_{t\in\mathbb{R}}$ an arbitrary smooth variation of $\phi$ and taking the derivative with respect to $t$ of the two c-bienergy functionals, we get
$$
\int_M\langle\tau_2^c\left(\tilde{\phi}\right), V\rangle \ v_{\tilde{g}} =\int_M\langle  \tau_2^c\left(\phi\right), V\rangle \ v_{g}, 
$$
i.e.,
$$
\int_M\langle e^{4\rho}\tau_2^c\left(\tilde{\phi}\right)-\tau_2^c\left(\phi\right), V\rangle \ v_{g}=0. 
$$
Now, we can conclude.
\end{proof}

\begin{Bem}	
When $M^4$ is non-compact, for any compact regular domain $D$ and any smooth variation $\left\{\phi_t\right\}_{t\in\mathbb{R}}$ of $\phi$ supported in $D$ we have
$$
E_2^c\left(\tilde{\phi}_t; D\right) = E_2^c\left(\phi_t; D\right)+2\int_{\partial D} \langle Y_{4,t},\eta\rangle \ v_{g} + \int_{\partial D} \langle Y_{5,t},\eta\rangle \ v_{g},
$$
where $Y_{5,t}=\left|d\phi_t\right|^2\nabla \rho$ and $\eta=\eta(p)$ denotes the outward pointing unit normal at a point $p\in\partial D$.

Taking the derivative with respect to $t$ in the above relation, we get
$$
\int_D\langle\tau_2^c\left(\tilde{\phi}\right), V\rangle \ v_{\tilde{g}} =\int_D\langle  \tau_2^c\left(\phi\right), V\rangle \ v_{g}
$$
and therefore $\tau_2^c\left(\tilde{\phi}\right)=e^{-4\rho}\tau_2^c(\phi)$ on $M^4$.
\end{Bem}

\begin{Bem}
We note that if $\phi:\mathbb{S}^4\to N^{n}$ is a c-biharmonic map then, removing the North pole of the sphere we obtain that $\tilde{\phi}:\left(\mathbb{S}^4\setminus\left\{\text{North pole}\right\}, \tilde{g}\right)\to N$, $\tilde{\phi}(p)=\phi(p)$, for any $p\in \mathbb{S}^4\setminus\left\{\text{North pole}\right\}$, is biharmonic, where $\tilde{g}$ is the flat metric induced by the stereographic projection $\Pi$. Equivalently, $\phi\circ \Pi^{-1}:\mathbb{R}^4\to N^n$ is biharmonic. A converse of this remark was proved in  \cite[Lemma 3.4]{MR2094320} by adding some integrability conditions. More precisely, the author showed that if $\varphi\in C^{\infty}\cap W^{2,2}\left(\mathbb{R}^4,\mathbb{S}^n\right)$ is a biharmonic map, then   $\phi=\varphi\circ\Pi:\mathbb{S}^4\to\mathbb{S}^{n}$ is a smooth c-biharmonic map defined on the whole of $\mathbb{S}^4$ (the isolated singularity was removed).
\end{Bem}

\begin{Cor}
Let $\sigma:M^4\to N^4$ be a conformal diffeomorphism and consider $\varphi:N^4\to P^n$ a smooth map. Assume that $M$ and $N$ are Ricci-flat. Then, $\phi=\varphi\circ\sigma:M^4\to P^n$ is biharmonic if and only if $\varphi$ is biharmonic.
\end{Cor}

As an example, we can consider the conformal inversion of $\mathbb{R}^4$ at $\mathbb{S}^3$ given by $\sigma:\mathbb{R}^4\setminus\left\{\bar{0}\right\}\to \mathbb{R}^4\setminus\left\{\bar{0}\right\}$, $\sigma\left(\bar{x}\right)=\bar{x}/\left|\bar{x}\right|^2$, and let $\phi: \mathbb{R}^4\setminus\left\{\bar{0}\right\}\to N^n $ be a smooth map. Then, we conclude that $\phi\circ \sigma$ is biharmonic if and only if $\phi$ is biharmonic (see \cite{MR3357596} and \cite{MR4110268}). We note that $\sigma$ is biharmonic too (see \cite{MR1952859}).


\begin{Cor}
Let $\sigma:M^4\to N^4$ be a conformal diffeomorphism and consider $\varphi:N^4\to P^n$ a harmonic map. Assume that $N$ is Ricci-flat. Then, $\phi=\varphi\circ\sigma:M^4\to P^n$ is c-biharmonic.
\end{Cor}

\begin{Cor}\label{prop:conformal-map}
Let	$\phi:M^4\to N^4$ be a conformal diffeomorphism. Then, $\phi$ is c-biharmonic if and only if the scalar curvature of $N$ is constant. In particular, any conformal diffeomorphism between four-dimensional space forms is c-biharmonic.
\end{Cor}

\begin{Bem}
Biharmonic conformal diffeomorphisms between four-dimensional space forms were completely classified in \cite{MR4610300}, and the only such maps are M\"obius transformations from flat manifolds.
\end{Bem}

Using the classical results of Yamabe, Trudinger, Aubin and Schoen (for a comprehensive survey on this topic see \cite{MR888880}) we immediately infer

\begin{Cor} 
Let $\left(M^4,g\right)$ be a compact manifold. Then, there exists a metric $\tilde{g}$ conformally equivalent to $g$ such that $\Id:\left(M^m,g\right)\to\left(M^m,\tilde{g}\right)$ is c-biharmonic. In the particular case of the unit Euclidean sphere $\left(\mathbb{S}^4,g\right)$, the metric $\tilde{g}$ is given by $\phi^{*}g$ multiplied by a positive constant, where $\phi$ is a global conformal diffeomorphism of $\left(\mathbb{S}^4,g\right)$; when $\left(M^4,g\right)$ is Ricci-flat, the metric $\tilde{g}$ is just the metric $g$ multiplied by a positive constant.
\end{Cor}

\begin{Bem}
The biharmonicity of conformal diffeomorphisms $\phi:M^4\to N^4$, where $M$ is an Einstein manifold, was studied in \cite{MR3881994} and an interesting link between biharmonicity and a certain Yamabe-type problem was established. 
\end{Bem}

As an application of the above results we can give a further example.

\begin{Bsp}
Let 
$$
\phi_{\zeta}: M^4_{\alpha}=\left(\mathbb{S}^3\times I, g=\alpha^2(r)g_{\mathbb{S}^3}+dr^2\right)\to 
N^{4}_{\beta}=\left(\mathbb{S}^3\times J, h=\beta^2(\zeta)g_{\mathbb{S}^3}+d\zeta^2\right)
$$
be a conformal rotational symmetric diffeomorphism defined by
$$
\phi_{\zeta}(\theta,r)=(\theta,\zeta(r)), \qquad \forall \theta\in\mathbb{S}^3,\ r\in I, 
$$
where $I$ and $J$ are some real open intervals and $\zeta:I\to J$ is a smooth function. Then, $\phi_\zeta$ is c-biharmonic if and only if the scalar curvature of $N^{4}_{\beta}$ is constant, i.e., $\beta$ is a solution of

\begin{equation}\label{eq-beta}
2\left(\beta'\right)^3-2\beta'-\beta\beta'\beta''-\beta^2\beta'''= 0,
\end{equation}
where $\beta'$ denotes the derivative of $\beta$ with respect to $\zeta$.

This equation can be deduced either using that the scalar curvature of $N^{4}_{\beta}$ is constant, or by a direct computation from $\tau_2^c\left(\phi_{\zeta}\right)=0$.

Of course, $\beta(\zeta)=\zeta$, $\beta(\zeta)=(1/d)\sin(d\zeta)$ and $\beta (\zeta)=(1/d)\sinh(d\zeta)$, where $d$ is a positive constant, satisfy the above equation, as they correspond to the  (whole) four-dimensional Euclidean space, Euclidean sphere and hyperbolic space, respectively. Moreover, a first integral of \eqref{eq-beta} is given by 
$$
\beta^2\left(1-\beta'^2+\beta\beta''\right)=a,\qquad a\in \mathbb{R},
$$
(see \cite{MR3357596} and \cite{MR3577677}). If $a\neq 0$, the solution $\beta$ provides a space of constant scalar curvature, but with non-constant sectional curvature. If $a=0$, the solution $\beta$ produces a space of constant sectional curvature and, further, if we impose the boundary condition $\beta(0)=0$ we obtain the above mentioned explicit solutions.

The conformal rotational symmetric diffeomorphism from $\mathbb{R}^4$ in $\mathbb{R}^4$, $\mathbb{S}^4$ or $\mathbb{H}^4$ can be determined by solving the equation
$$
\dot{\zeta}=\pm \frac{\beta}{r},
$$
where $\dot{\zeta}$ denotes the derivative of $\zeta$ with respect to $r$. The solutions of the above equation are biharmonic (and c-biharmonic) maps and they extend smoothly over the origin of $\mathbb{R}^4$.

In order to find the conformal rotational symmetric diffeomorphism from $\mathbb{S}^4$ to $\mathbb{R}^4$, $\mathbb{S}^4$ or $\mathbb{H}^4$, we can either solve the equation
$$
\dot{\zeta}=\pm \frac{\beta}{\sin r},
$$
or compose all conformal rotational symmetric diffeomorphisms from $\mathbb{R}^4$  with the stereographic projection of $\mathbb{S}^4$ in $\mathbb{R}^4$. These maps are c-biharmonic and smooth on the whole of $\mathbb{S}^4$. 

We also mention that all biharmonic rotational symmetric maps from $\mathbb{R}^4$ to $\mathbb{R}^4$, $\mathbb{S}^4$ or $\mathbb{H}^4$, which are not necessarily conformal, provide c-biharmonic rotationally symmetric maps from $\mathbb{S}^4$, possible without the North or South pole, in $\mathbb{R}^4$, $\mathbb{S}^4$ or $\mathbb{H}^4$. In particular, the harmonic rotationally symmetric maps from $\mathbb{R}^4$ in $\mathbb{S}^4$ of Jäger and Kaul (see 
\cite{MR705882}), provide (non-harmonic) c-biharmonic rotationally symmetric maps from $\mathbb{S}^4$, possibly without the North or South pole, in $\mathbb{S}^4$.
\end{Bsp}  

\section{c-biharmonic hypersurfaces in space forms}

Let $\phi:M^m\to N^{m+1}(c)$ be a hypersurface in a space form of constant sectional curvature $c$. We denote by $\eta$ the unit section in the normal bundle of $M^m$ in $N^{m+1}(c)$, by $A=A_\eta$ the shape operator of $M$ and by $f=\tr A/m$ the mean curvature function of $M$. When $f$ is constant, we say that $M$ is a CMC (constant mean curvature, including minimal) hypersurface. 

We recall (see for example \cite[Section 2.4]{MR4410183}) 
that, for any hypersurface in a space form, the bitension field looks like
$$
\frac{1}{m}\tau_2(\phi)= \left(-\Delta f -f|A|^2+mcf\right)\eta-mf\nabla f-2A(\nabla f).
$$
Therefore, $\tau_2^c(\phi)$ takes the following form
\begin{align}\label{tau2c-tangent-part-normal-part}
\frac{1}{m}\tau_2^c(\phi)= &\left(-\Delta f -f|A|^2+mcf-\frac{2}{3}\Scal f+\frac{2}{m}\langle A, \Ric\rangle\right)\eta \nonumber\\
&  -mf\nabla f-2A(\nabla f)+\frac{1}{3m}\nabla \Scal.
\end{align}

Since for any hypersurface in a space form $N^{m+1}(c)$ we have
\begin{equation}\label{ric}
\Ric(X)=c(m-1)X+mfA(X)-A^2(X), \qquad \forall X\in C(TM),
\end{equation}
and 
\begin{equation}\label{scal}
\Scal=c(m-1)m+m^2f^2-|A|^2,
\end{equation}
the above formula can be rewritten as

\begin{align}\label{tau2c-spaceforms} 
	\frac{1}{m}\tau_2^c(\phi)= &\left(-\Delta f +\frac{f}{3}\left(5|A|^2-2m^2f^2-c\left(2m^2-11m+6\right)\right)-\frac{2}{m}\tr A^3\right)\eta \nonumber\\
	& -mf\nabla f-2A(\nabla f)+\frac{1}{3m}\nabla \Scal.
\end{align}

We directly infer 

\begin{Prop}
Any totally geodesic hypersurface in $N^{m+1}(c)$ is c-biharmonic.
\end{Prop}

\begin{Prop}
Let $\phi:M^m\to N^{m+1}(c)$ be a minimal hypersurface. Then, $M$ is c-biharmonic if and only if $\tr A^3=0$ and it has constant scalar curvature.
\end{Prop}

When $M^m$, $m\geq 2$, is a minimal hypersurface in a unit Euclidean sphere, we can establish an interesting link between the c-biharmonic hypersurfaces and the well-known Willmore hypersurfaces. We just recall here that the Willmore hypersurfaces in Euclidean spheres are the critical points of the Willmore functional
$$
\mathcal{W}(\phi)=\left(\frac{m}{m-1}\right)^{m/2} \int_M\left(|A|^2-mf^2\right)^{m/2} \ v_g,
$$
which is invariant under the conformal transformations of $\mathbb{S}^{m+1}$. Using Moebius invariants, the first variation was given by Wang in \cite{MR1639852}. The Euler-Lagrange equation can be written as
$$
\rho^{m-2}\left(2f|A|^2-mf^3-\tr A^3\right)+(m-1)\Delta \left(\rho^{m-2}f\right)+\langle \Hess \rho^{m-2},mf\Id-A\rangle=0,
$$
where $\rho^2=|A|^2-mf^2$ (see also \cite{MR2150692}).
\begin{Theorem}\label{thm-Willmore}
Let $\phi:M^m\to \mathbb{S}^{m+1}$, $m\geq 2$, be a minimal hypersurface with $A\neq 0$. Then,  $M$ is c-biharmonic if and only if it is a Willmore hypersurface with constant scalar curvature.
\end{Theorem}

\begin{Bem}\label{remark-4-curvatures}
When $m=4$, minimal Willmore hypersurfaces with constant scalar curvature were studied in \cite{MR2150692}. From Theorem $1$ in \cite{MR2150692}, we conclude that there are examples of minimal c-biharmonic hypersurfaces (which are not Einstein) with $4$ constant distinct principal curvatures: $1+\sqrt{2}$, $1-\sqrt{2}$, $-1+\sqrt{2}$ and $-1-\sqrt{2}$. Analogously, when $m=6$, according to Theorem $5.1$ in \cite{MR1868938}, there are examples of minimal c-biharmonic hypersurfaces (which are not Einstein) with $6$ constant distinct principal curvatures: $2+\sqrt{3}$, $2-\sqrt{3}$, $1$, $-1$, $-2+\sqrt{3}$ and $-2-\sqrt{3}$. Moreover, very recently, in \cite{BMNOR} it was proved that there exist non-minimal c-biharmonic hypersurfaces in Euclidean spheres with $4$ constant distinct principal curvatures. Recall that, in the biharmonic case, all known examples of proper biharmonic hypersurfaces have only $1$ or $2$ constant distinct principal curvatures: $1$ or $-1$ and $1$, respectively. 
\end{Bem}

We also see that if $M^m$ is a c-biharmonic hypersurface in a space form and has constant scalar curvature, then it is biconservative, i.e., the tangent part of the bitension field $\tau_2(\phi)$ vanishes (see for example \cite{MR4410183}). Using the results in \cite{MR4215275} and \cite{MR4903596}  we obtain 

\begin{Theorem}
Let $\phi:M^m\to N^{m+1}(c)$ be a compact c-biharmonic hypersurface. If $M$ has constant scalar curvature and non-negative sectional curvature, then $\nabla A=0$.
\end{Theorem}

\begin{Theorem}
Let $\phi:M^3\to N^{4}(c)$ be a c-biharmonic hypersurface. If $M$ has constant scalar curvature, then it is CMC.
\end{Theorem}

Thus, it seems natural to study CMC c-biharmonic hypersurfaces in $N^{m+1}(c)$. From \eqref{tau2c-spaceforms} we obtain

\begin{Prop}
Let $\phi:M^m\to N^{m+1}(c)$ be a CMC hypersurface. Then, $M$ is c-biharmonic if and only if
\begin{equation*}
\left\{
\begin{array}{l}
mf\left(5|A|^2-2m^2f^2-c\left(2m^2-11m+6\right)\right)-6\tr A^3 = 0, \\\\
\nabla \Scal = 0.
\end{array}
\right.
\end{equation*}
\end{Prop}

Using the known formula which holds for any CMC hypersurface in a space form (see \cite{MR0266109})
\begin{equation}\label{nomitzu}
m f \tr A^3=-|\nabla A|^2+c m^2 f^2+|A|^2\left(|A|^2-c m\right)-\frac{1}{2}\Delta|A|^2,
\end{equation}
we deduce

\begin{Prop}
Let	$\phi:M^m\to N^{m+1}(c)$ be a non-minimal CMC c-biharmonic hypersurface. Then, $M$ has constant scalar curvature and
\begin{equation}\label{inequality-scalar curvature}
0\leq 6|\nabla A|^2=|A|^2\left(6|A|^2-5m^2f^2-6c m\right)+m^2f^2\left(2m^2f^2+c\left(2m^2-11m+12\right)\right).
\end{equation}
Equality holds if and only if $\nabla A=0$.
\end{Prop}

In fact, the equality case in \eqref{inequality-scalar curvature} is a characterization formula for c-biharmonicity when $M^m$ is a non-minimal hypersurface in a space form with $\nabla A=0$. Further, since for any hypersurface in a space form,  $\nabla A=0$ is equivalent to $f$ is constant and $\nabla \Ric=0$ (see \cite{MR0238229} and \cite{MR0296859}), we can state
\begin{Prop}

Let	$\phi:M^m\to N^{m+1}(c)$ be a non-minimal CMC hypersurface. If $\nabla \Ric=0$, then $M$ is c-biharmonic if and only if 	
\begin{equation}\label{c-biharmonicity-Einstein}
|A|^2\left(6|A|^2-5m^2f^2-6c m\right)+m^2f^2\left(2m^2f^2+c\left(2m^2-11m+12\right)\right)=0.
\end{equation}
\end{Prop} 

We know that a minimal Einstein hypersurface is c-biharmonic. When the hypersurface is non-minimal, the situation is more rigid.

\begin{Theorem}\label{c-biharmonic-umbilical}
Let	$\phi:M^m\to N^{m+1}(c)$ be a non-minimal CMC Einstein hypersurface with $\Ric=\lambda\Id$, where $\lambda$ is a real constant. Then, $M$ is c-biharmonic if and only if $M$ is umbilical,
$$
\lambda=\frac{6mc(m-1)}{2m^2-5m+6},\qquad f^2=\frac{c\left(-2m^2+11m-6\right)}{2m^2-5m+6}, \qquad c\neq 0.
$$ 
In this case, one of the following holds
\begin{itemize}
	\item[i)] $c>0$ and $m\leq 4$;
	\item[ii)] $c<0$ and $m\geq 5$.
\end{itemize}
\end{Theorem}  

\begin{proof}
Since $M^m$ is a non-minimal Einstein hypersurface in $N^{m+1}(c)$, from \eqref{ric}, taking the inner product with $A$, and then the trace, we obtain

\begin{equation}\label{similar-nomitzu}
mf\tr A^3=m^2f^2\left(c(m-1)-\lambda+|A|^2\right).
\end{equation} 

Replacing \eqref{similar-nomitzu} in \eqref{nomitzu}, as $M$ is also CMC with $|A|^2$ constant and $\nabla A=0$, we get

$$
m^2f^2\left(|A|^2-\lambda+c(m-2)\right)=|A|^2\left(|A|^2-cm\right).
$$ 

Next, from \eqref{scal} we see that

\begin{equation*}\label{norm-A^2}
|A|^2=m^2f^2-m\left(\lambda-c(m-1)\right).
\end{equation*}

From the last two relations we achieve

$$
\left(\lambda-c(m-2)\right)\left((m-1)f^2-(\lambda-c(m-1))\right)=0.
$$

Further, the classification of non-minimal CMC Einstein hypersurfaces in space forms, when $m\geq 2$, can be done using standard arguments according to $\lambda=c(m-2)$ or $\lambda \neq c(m-2)$ (see also \cite{MR0296859}). Therefore, every non-minimal CMC Einstein hypersurface $M^m$, $m\geq 2$, in a space form $N^{m+1}(c)$, with $\Ric=\lambda\Id$, where $\lambda$ is a real constant, falls into one of the following cases. 

\medskip\noindent
\textbf{Case A:} $\lambda=c(m-2)$. Then, $|A|^2=m\left(mf^2+c\right)$
and three subcases occur according to the sign of $c$:

\begin{itemize}
	\item if $c<0$, then either $M$ is umbilical and $f^2=c/(1-m)$, or $M$ is not umbilical, $m=2$ and it has $2$ constant distinct principal curvatures $f\pm\sqrt{f^2+c}$.
	\item if $c=0$, then $M$ is not umbilical and has $2$ constant distinct principal curvatures $0$ and $mf$, with multiplicities $m-1$ and $1$, respectively.
	\item if $c>0$, then $M$ is not umbilical and has $2$ constant distinct principal curvatures $\left(mf\pm\sqrt{m^2f^2+4c}\right)/2$.
\end{itemize}

\medskip\noindent
\textbf{Case B:} $\lambda\neq c(m-2)$. Then, $|A|^2=mf^2$, i.e., $M$ is umbilical, and
$$
f^2=\frac{\lambda-c(m-1)}{m-1}.
$$

From \eqref{tau2c-tangent-part-normal-part} we deduce that a non-minimal CMC Einstein hypersurface in a space form is c-biharmonic if and only if 

\begin{equation}\label{c-bihar-Einstein}
-|A|^2+mc-\frac{2}{3}\left(c(m-1)m+m^2f^2-|A|^2\right)+2\lambda=0.
\end{equation}

Substituting the data from Case A into \eqref{c-bihar-Einstein} yields no solutions; hence the c-biharmonic condition cannot occur in Case A. If we consider the Case B, replacing $|A|^2=mf^2$ and $f^2=(\lambda-c(m-1))/(m-1)$ in \eqref{c-bihar-Einstein} we obtain

\begin{equation}\label{value-lambda}
\lambda=\frac{6mc(m-1)}{2m^2-5m+6}
\end{equation}

and

\begin{equation}\label{value-f^2}
f^2=\frac{c\left(-2m^2+11m-6\right)}{2m^2-5m+6}.
\end{equation}

Finally, since $f^2>0$, we must have $c\left(-2m^2+11m-6\right)>0$. As $-2m^2+11m-6>0$ precisely for $2\le m\le 4$, we conclude that either $c>0$ and $m\le 4$, or $c<0$ and $m\ge 5$. 

In order to complete the proof, we need to consider also the case $m=1$. Since, in this case, $M$ is umbilical and $\lambda=0$, we can check that \eqref{c-bihar-Einstein} is equivalent to $f^2=c$ and so $c>0$. Therefore, relations \eqref{value-lambda} and $\eqref{value-f^2}$ hold also for $m=1$.

\end{proof}

The following statement provides a rigidity phenomenon showing that, under a concrete pinching condition on $f$ and $m$, any non-minimal CMC c‑biharmonic hypersurface in $N^{m+1}(c)$ which is also biharmonic necessarily satisfies $\nabla A=0$.

\begin{Prop}\label{c-biharmonic-biharmonic}
Let	$\phi:M^m\to N^{m+1}(c)$ be a non-minimal CMC c-biharmonic hypersurface. Assume that $M$ is biharmonic, $0<f^2\leq  c(m-2k)^2/m^2$, for some non-negative integer $k$ such that $2k+1\leq m$, and $m\in\left[m_1,m_2\right]$, where $m_1,m_2$ are the real roots of
\begin{equation}\label{second-degree-equation}
m^2-2(k+2)m+2k^2+3=0.
\end{equation}
Then,
$$
c>0, \quad \nabla A=0, \quad f^2=\frac{c(m-2k)^2}{m^2} \quad \text{and} \quad (k,m)\in \left\{(0,1),(0,3),(1,5), (3,7)\right\}.
$$ 
\end{Prop}

\begin{proof}
Recall that (see for example \cite{MR4410183}), a non-minimal CMC hypersurface in $N^{m+1}(c)$ is biharmonic if and only if $|A|^2=cm$. Thus, $c>0$. 

On the other hand, from \eqref{inequality-scalar curvature} and our hypothesis, we have
\begin{equation}\label{inequalities-for-nabla-A}
0\leq\frac{3}{m^2f^2}|\nabla A|^2=m^2f^2+c\left(m^2-8m+6\right)\leq 2c\left(m^2-2(k+2)m+2k^2+3\right)\leq 0.
\end{equation}
Since \eqref{second-degree-equation} has two real roots, we obtain $k\in\left\{0,1,2,3,4\right\}$, and from \eqref{inequalities-for-nabla-A} we conclude by direct computations.

\end{proof}

In the following we study the c-biharmonicity of hypersurfaces in $N^{m+1}(c)$ with $\nabla A=0$, i.e., with $1$ or $2$ constant distinct principal curvatures.

\subsection{c-biharmonic hypersurfaces in Euclidean spaces}
It is not difficult to check that the hyperplanes of $\mathbb{R}^{m+1}$ are totally geodesic and c-biharmonic and no $m$-dimensional hypersphere or cylinder is c-biharmonic in $\mathbb{R}^{m+1}$.

\subsection{c-biharmonic hypersurfaces in Euclidean spheres}
First, we consider the canonical inclusion $\iota\colon\s^m(r)\to\s^{m+1}$ of the hypersphere $\mathbb{S}^m(r)$ in $\mathbb{S}^{m+1}$. For
\begin{align*}
\mathbb{S}^{m}(r) &=\mathbb{S}^{m}(r)\times\left\{\sqrt{1-r^2}\right\} \\
&= \left\{\left(\overline{x},\sqrt{1-r^2}\right)\in\mathbb{R}^{m+2}\ |\  \left|\overline{x}\right|^2=r^2,\ 0<r\leq 1\right\},
\end{align*}
we can choose $\eta=\left(\left(\sqrt{1-r^2}/r\right)\overline{x},-r\right)$ and thus the shape operator is given by $A=A_\eta=-\left(\sqrt{1-r^2}/r\right)\Id$. 

It is not difficult to check that
$$
H=-\frac{\sqrt{1-r^2}}{r}\eta, \qquad \Ric=\frac{m-1}{r^2}\Id, \qquad \Scal=\frac{m(m-1)}{r^2}.
$$
Then, by straightforward computations, we obtain the tension and bitension fields
$$
\tau(\iota)=-\frac{m\sqrt{1-r^2}}{r}\eta, \qquad
\tau_2(\iota)=\frac{m^2\sqrt{1-r^2}\left(1-2r^2\right)}{r^3} \eta,
$$
and
$$
\langle A, \Ric\rangle=-\frac{m(m-1)\sqrt{1-r^2}}{r^3}.
$$
Consequently, we obtain the formula for $\tau_2^c (\iota)$ as follows
\begin{equation}\label{eq:tau2c-first-example}
\tau_2^c(\iota)=\frac{\sqrt{1-r^2}}{3r^3}\left(-6mr^2+2m^2-5m+6\right)m\eta.
\end{equation}

Using Proposition \ref{prop:restriction-of-radius}, we get the following lower bound for the radius
\begin{align*}
	r^2>\frac{2}{3}\frac{(m-1)(m-3)}{m}.
\end{align*}
When $r^2<1$, we immediately get the restriction $m\leq 4$, which agrees with Theorem \ref{c-biharmonic-umbilical}.  Now, we can state the following classification result

\begin{Theorem}\label{thm:inclusion-sphere}
The hypersphere $\s^m(r)$ is c-biharmonic in $\s^{m+1}$ if and only if
	\begin{enumerate}
		\item [i)] $r=1$, i.e., $\s^m(r)$ is totally geodesic in $\s^{m+1}$ and for any $m\geq 1$;
		\item [ii)] $m=1$ or $m=3$ and $r=1/\sqrt{2}$;
		\item [iii)] $m=2$ and $r=1/\sqrt{3}$;
		\item [iv)] $m=4$ and $r=\sqrt{3}/2$.
	\end{enumerate}
\end{Theorem}
\begin{proof}

From \eqref{eq:tau2c-first-example} we can see that $\tau_2^c(\iota)=0$ is equivalent to either $r=1$ corresponding to the totally geodesic case, or
$$
r^2=\frac{2m^2-5m+6}{6m},
$$ 
which gives rise to the other three cases.
\end{proof}

\begin{Bem}
Removing the North pole of the c-biharmonic hypersurface $\mathbb{S}^4\left(\sqrt{3}/2\right)$, the canonical inclusion $\tilde{\iota}:\left(\mathbb{S}^4(\sqrt{3}/2)\setminus\left\{\text{North pole}\right\}, \tilde{g}\right)\to\mathbb{S}^5$ is a biharmonic immersion which is not an isometric immersion but a conformal immersion, where $\tilde{g}$ is the flat metric induced by the stereographic projection. Or, we can say that $\iota\circ \Pi^{-1}:\mathbb{R}^4\to \mathbb{S}^5$ is a $SO(4)$-equivariant biharmonic map, where the action of $SO(4)$ on $\mathbb{S}^5$ fixes the last two coordinates of $\mathbb{R}^6$.
\end{Bem}

For hypersurfaces with $2$ constant distinct principal curvatures, we consider the canonical inclusion $\iota:\mathbb{S}^{m_1}\left(r_1\right)\times \mathbb{S}^{m_2}\left(r_2\right)\to\mathbb{S}^{m+1}$ of the generalized Clifford torus $\mathbb{S}^{m_1}\left(r_1\right)\times \mathbb{S}^{m_2}\left(r_2\right)$ in $\mathbb{S}^{m+1}$, where $r_1$, $r_2$ are positive real numbers such that $r_1^2+r_2^2=1$ and $m_1$, $m_2$ are positive integers, $m_1+m_2=m$. For
\begin{equation*}
\mathbb{S}^{m_1}\left(r_1\right)\times \mathbb{S}^{m_2}\left(r_2\right) = \left\{\left(\overline{x}_1,\overline{x}_2\right)\in\mathbb{R}^{m_1+1}\times\mathbb{R}^{m_2+1} |\  \left|\overline{x}_1\right|^2=r_1^2, \ \left|\overline{x}_2\right|^2=r_2^2, \  r_1^2+r_2^2=1 \right\},
\end{equation*}
we can choose $\eta=\left(\left(r_2/r_1\right)\overline{x}_1,-\left(r_1/r_2\right)\overline{x}_2\right)$ and thus, the shape operator is given by 
$$
A\left(X_1\right)=-\frac{r_2}{r_1}X_1 \qquad \text{and}\qquad A\left(X_2\right)=\frac{r_1}{r_2} X_2,
$$
for any $X_1\in C\left(T\mathbb{S}^{m_1}\left(r_1\right)\right)$ and $X_2\in C\left(T\mathbb{S}^{m_2}\left(r_2\right)\right)$.
We recall that
\begin{align*}
H&=\frac{1}{m}\left(-\frac{r_2}{r_1}m_1+\frac{r_1}{r_2}m_2\right)\eta, \\
\Ric\left(X_i\right)&=\frac{1}{r_i^2}\left(m_i-1\right)X_i, \qquad \forall X_i\in C\left(T\mathbb{S}^{m_i}\left(r_i\right)\right), i\in\{1,2\} 
\end{align*}
and
\begin{align*}
\Scal=\frac{1}{r_1^2}m_1\left(m_1-1\right)+\frac{1}{r_2^2}m_2\left(m_2-1\right).
\end{align*}
By standard computations, we get 
\begin{align*}
\tau(\iota)=& \left(-\frac{r_2}{r_1}m_1+\frac{r_1}{r_2}m_2\right)\eta,\\
\tau_2(\iota)=&\left(-\frac{r_2}{r_1}m_1+\frac{r_1}{r_2}m_2\right)
\left(\left(1-\frac{r_2^2}{r_1^2}\right)m_1+\left(1-\frac{r_1^2}{r_2^2}\right)m_2\right)\eta
\end{align*}
and
\begin{align*}
\langle A,\Ric\rangle = -\frac{r_2}{r_1^3}m_1\left(m_1-1\right)+\frac{r_1}{r_2^3}m_2\left(m_2-1\right).
\end{align*}
From \eqref{eq-c-biharmonic-intro} or \eqref{tau2c-spaceforms}, we get the formula
for the conformal bitension field
\begin{align}\label{eq:tau2c-clifford}
\tau_2^c(\iota) = &\left(  \left(\frac{r_2}{r_1}m_1-\frac{r_1}{r_2}m_2\right)\left(\left(r_2^2-r_1^2\right)\left(\frac{1}{r_1^2}m_1-\frac{1}{r_2^2}m_2\right)
+\frac{2}{3}\sum_{i=1}^2\frac{1}{r_i^2}(m_i-1)(m_i-3)\right)\right. \nonumber \\
&\quad \left.-\frac{2}{r_1r_2}\left(m_1-m_2\right)\vphantom{\frac{2}{3}\sum_{i=1}^2\frac{1}{r_i^2}(m_i-1)(m_i-3)}\right) \eta.
\end{align}

First, we consider the simplest case, i.e., when $m_1=m_2$.

\begin{Prop}
\label{thm:c-biharmonic-clifford}
The generalized Clifford torus $\s^{m/2}\left(r_1\right)\times \s^{m/2}\left(r_2\right)$, where 
$m$ is even and $r_1^2+r_2^2=1$, is c-biharmonic in $\s^{m+1}$ if and only if
\begin{itemize}
	\item [i)]  $r_1^2=r_2^2=1/2$, i.e., $\s^{m/2}(r_1)\times S^{m/2}(r_2)$ is
	minimal in $\s^{m+1}$;
	\item [ii)] $m=4$ and $r_1^2=\frac{1}{2}\left(1-\frac{1}{\sqrt{3}}\right)$, $r_2^2=\frac{1}{2}\left(1+\frac{1}{\sqrt{3}}\right)$.
\end{itemize}
\end{Prop}

\begin{proof}
In the case $m_1=m_2$, using \eqref{eq:tau2c-clifford}, equation $\tau_2^c(\iota)=0$ simplifies to
\begin{align*}
	\left(r_2^2-r_1^2\right)
	\left(\left(r_2^2-r_1^2\right)^2m_1
	+\frac{2}{3}\left(m_1-1\right)\left(m_1-3\right)
	\right)=0.
\end{align*}
Taking into account the constraint $r_1^2+r_2^2=1$, a solution of the above equation is
\begin{align*}
	r_1=r_2=\frac{1}{\sqrt{2}},
\end{align*}
which corresponds to the minimal generalized Clifford torus. Further, we note that the equation  
\begin{align*}
\left(r_2^2-r_1^2\right)^2m_1+\frac{2}{3}\left(m_1-1\right)\left(m_1-3\right)=0
\end{align*}
forces
\begin{align*}
\left(m_1-1\right)\left(m_1-3\right)<0,
\end{align*}
which can be satisfied only if $m_1=2$, i.e., $m=4$. Therefore, we get the system
\begin{align*}
	r_2^2-r_1^2=\pm\frac{1}{\sqrt{3}}\qquad \text{and} \qquad r_1^2+r_2^2=1
\end{align*}
and obtain the non-minimal solution of $\tau_2^c(\iota)=0$.
\end{proof}

Second, if we consider $m_1\neq m_2$ and $r_1=r_2=1/\sqrt{2}$, by a direct computation or from Proposition \ref{c-biharmonic-biharmonic}, we get

\begin{Prop}
The generalized Clifford torus $\s^{m_1}\left(1/\sqrt{2}\right)\times \s^{m_2}\left(1/\sqrt{2}\right)$, where $m_1\neq m_2$, $m_1+m_2=m$, is c-biharmonic in $\s^{m+1}$ if and only if
\begin{itemize}
\item [i)] $m_1=1$ and $m_2=4$, and thus $m=5$;
\item [ii)] $m_1=3$ and $m_2=4$, and thus $m=7$.
\end{itemize}
\end{Prop}

In fact, we can give the following complete classification of c-biharmonic generalized Clifford tori.
 
\begin{Theorem}
\label{thm:c-biharmonic-clifford-complete}
The generalized Clifford torus $\s^{m_1}\left(r_1\right)\times \s^{m_2}\left(r_2\right)$, where $m_1+m_2=m$ and $r_1^2+r_2^2=1$, is c-biharmonic in $\s^{m+1}$ if and only if one of the following cases holds
\begin{itemize}
	\item[i)] $m_1=m_2=2$ and either 
	$$
	r_1^2=r_2^2=\frac{1}{2}, \qquad\text{or}\qquad r_1^2=\frac{1}{2}\left(1-\frac{1}{\sqrt{3}}\right),\  r_2^2=\frac{1}{2}\left(1+\frac{1}{\sqrt{3}}\right).
	$$
	\item[ii)] $m_1\neq 2$ or $m_2\neq 2$ and 
	$$
	r_1^2=\frac{T_\ast}{1+T_\ast}, \ r_2^2=\frac{1}{1+T_\ast},
	$$
	where $T_\ast$ is the unique positive solution of the polynomial equation
	\begin{equation}\label{pol-grad3}
		a_3T^3 + a_2T^2+a_1T+a_0=0,
	\end{equation}
	with the coefficients given by
	\begin{equation*}
		\left\{
		\begin{array}{l}
			a_0=m_1\left(2m_1^2-5m_1+6\right)\\\\
			a_1=m_1\left(\left(m_1-m_2\right)^2+\left(m_1-\frac{11}{2}\right)^2+\left(m_2-3\right)^2-\frac{133}{4}\right)\\\\
			a_2=-m_2\left(\left(m_2-m_1\right)^2+\left(m_2-\frac{11}{2}\right)^2+\left(m_1-3\right)^2-\frac{133}{4}\right)\\\\
			a_3=-m_2\left(2m_2^2-5m_2+6\right)
		\end{array}
		\right..
	\end{equation*}	
\end{itemize}
\end{Theorem}

\begin{proof}
First, we can rewrite \eqref{eq:tau2c-clifford} as
\begin{align*}
	\tau_2^c(\iota) = &\left(  \left(\frac{r_2}{r_1}m_1-\frac{r_1}{r_2}m_2\right)\left(\left(\frac{r_2^2}{r_1^2}-1\right)m_1+\left(\frac{r_1^2}{r_2^2}-1\right)m_2
	+\frac{2}{3}\sum_{i=1}^2\frac{1}{r_i^2}(m_i-1)(m_i-3)\right)\right. \nonumber \\
	&\quad \left.-\frac{2}{r_1r_2}\left(m_1-m_2\right)\vphantom{\frac{2}{3}\sum_{i=1}^2\frac{1}{r_i^2}(m_i-1)(m_i-3)}\right) \eta.
\end{align*}
Since $r_1^2+r_2^2=1$, if we denote by $\alpha=r_1/r_2>0$, we have
$$
\frac{1}{r_1^2}=1+\frac{1}{\alpha^2}, \qquad \frac{1}{r_2^2}=1+\alpha^2, \qquad \frac{1}{r_1r_2}=\alpha+\frac{1}{\alpha}
$$
and
\begin{align*}
	\tau_2^c(\iota) = &\left( \left(\frac{1}{\alpha}m_1-\alpha m_2\right)\left(\left(\frac{1}{\alpha^2}-1\right)m_1+\left(\alpha^2-1\right)m_2 \right.\right. \nonumber\\
	 &\quad \left.\left. +\frac{2}{3}\left(\left(m_1-1\right)\left(m_1-3\right)\left(1+\frac{1}{\alpha^2}\right)+\left(m_2-1\right)
	\left(m_2-3\right)\left(1+\alpha^2\right)\right)\right)\right. \nonumber \\
	&\quad \left. -2\left(m_1-m_2\right)\left(\alpha+\frac{1}{\alpha}\right)\right) \eta.
\end{align*}
Imposing $\tau_2^c(\iota)=0$ and denoting $T=\alpha^2$, by some algebraic computations we achieve the polynomial equation \eqref{pol-grad3}.

It is not difficult to check that $a_0>0$, $a_1\neq 0$, $a_2\neq 0$ and  $a_3<0$ for any positive integers $m_1$, $m_2$. Using Descartes' rule of signs, we conclude that there always exists at least one positive solution of \eqref{pol-grad3}. More precisely, according to the signs of $a_1$ and $a_2$, we have exactly one positive solution in all cases, with the exception when $a_1<0$ and $a_2>0$. In this case, one can have $1$ or $3$ positive solutions (counting the multiplicities) of \eqref{pol-grad3}.

Further, we can prove that the conditions $a_1<0$ and $a_2>0$ imply $m_1$, $m_2\leq 8$. Now, using a computer algebra system, we consider all possibilities for $m_1$ and $m_2$, and we see that the only case when there exist $3$ positive solutions is $m_1=m_2=2$. In this case they are
$$
T_\ast= 1, \qquad T_\ast=2-\sqrt{3}, \qquad T_\ast=2+\sqrt{3}.
$$
Thus, we obtain
$$
r_1^2=\frac{1}{2}, \qquad r_1^2=\frac{1}{2}\left(1-\frac{1}{\sqrt{3}}\right),\qquad  r_1^2=\frac{1}{2}\left(1+\frac{1}{\sqrt{3}}\right).
$$
Since the classification is done up to isometries of the ambient sphere, we conclude.
\end{proof}

\subsection{c-biharmonic hypersurfaces in hyperbolic spaces}
For the hyperbolic space of curvature $c=-1$ we consider the one-sheet hyperboloid model lying  in the Minkowski space $\mathbb{R}^{n,1}$. More precisely, in $\mathbb{R}^{n+1}$, $n\geq 2$, we define the inner product
$$
\langle X, Y\rangle:=\sum_{i=1}^{n}X^iY^i-X^{n+1}Y^{n+1},
$$
where $X=\left(X^1,X^2,\dots, X^{n+1}\right)$ and $Y=\left(Y^1,Y^2,\dots, Y^{n+1}\right)$ are vectors in $\mathbb{R}^{n+1}$. The hyperbolic space is defined by
$$
\mathbb{H}^n=\left\{\bar{x}\in\mathbb{R}^{n+1}\ | \ \langle\bar{x},\bar{x}\rangle=-1\quad \text{and} \quad x^{n+1}>0\right\}.
$$
The geometry of hypersurfaces in hyperbolic space was first discussed by Ryan in \cite{MR0253243}, see also the correction in \cite{MR1940572}.

The hypersurfaces $M^m$ of $\mathbb{H}^{m+1}$, $m\geq 2$, with $1$ or $2$ constant distinct principal curvatures, are the following
\begin{itemize}
	\item [i)] $M^m=\left\{\bar{x}\in\mathbb{H}^{m+1}\ | \ x^{1}=r\geq 0\right\}$. As the unit section in the normal bundle of $M^m$ in $\mathbb{H}^{m+1}$ we can choose
	$$
	\eta=\left(\sqrt{1+r^2},\frac{r}{\sqrt{1+r^2}}\left(x^2, \dots, x^{m+2}\right)\right).
	$$ 
	We have
	$$
	 A =- \frac{r}{\sqrt{1+r^2}} \Id, \quad H= -\frac{r}{\sqrt{1+r^2}}\eta, \quad \Ric=-\frac{m-1}{1+r^2}\Id, \quad \Scal=-\frac{(m-1)m}{1+r^2};
	$$
	$M^m$ is isometric to the hyperbolic space of curvature $-1/\left(1+r^2\right)$.
	
	\item [ii)] $M^m=\left\{\bar{x}\in\mathbb{H}^{m+1}\ | \ x^{m+2}=x^{m+1}+a,\  a>0\right\}$. As the unit section in the normal bundle of $M^m$ in $\mathbb{H}^{m+1}$ we can choose
	$$
	\eta=\left(x^1, \dots, x^{m}, x^{m+1}-\frac{1}{a}, x^{m+1}+a-\frac{1}{a}\right).
	$$ 
	We have
	$$
	A=-\Id, \quad H= -\eta, \quad \Ric=0, \quad \Scal=0; 
	$$
	$M^m$ is isometric to the Euclidean space $\mathbb{R}^{m}$.
	
	\item[iii)] $M^m=\left\{\bar{x}\in\mathbb{H}^{m+1}\ | \ \sum_{i=1}^{m+1}\left(x^{i}\right)^2=r^2,\  r>0\right\}$. As the unit section in the normal bundle of $M^m$ in $\mathbb{H}^{m+1}$ we can choose
	$$
	\eta=\frac{\sqrt{1+r^2}}{r}\left(x^1, \dots, x^{m+1},\frac{r^2}{\sqrt{1+r^2}}\right).
	$$ 
	We have
	$$
	A= -\frac{\sqrt{1+r^2}}{r}\Id, \quad H=-\frac{\sqrt{1+r^2}}{r}\eta, \quad \Ric=\frac{m-1}{r^2}\Id, \quad \Scal=\frac{(m-1)m}{r^2}; 
	$$
	$M^m$ is isometric to the Euclidean sphere of radius $r$.
	
	\item[iv)] 
	\begin{align*}
	M^m_k=\left\{\left(\bar{x}_1,\bar{x}_2\right)\in\mathbb{H}^{m+1}\ \right.|&\left. \ \left|\bar{x}_1\right|^2= \sum_{i=1}^{k+1}\left(x^{i}\right)^2=r^2,\ 
	\left|\bar{x}_2\right|^2=\sum_{i=k+2}^{m+1}\left(x^{i}\right)^2 - \left(x^{m+2}\right)^2=-\left(1+r^2\right),\right. \\
	&\left.  r>0, 0\leq k\leq m-1\right\}.
	\end{align*}
	As the unit section in the normal bundle of $M^m_k$ in $\mathbb{H}^{m+1}$ we can choose
	$$
	\eta=\left(\frac{\sqrt{1+r^2}}{r}\bar{x}_1, \frac{r}{\sqrt{1+r^2}}\bar{x}_2\right)
	$$
	and therefore
	$$
	A\left(X_1\right)=- \frac{\sqrt{1+r^2}}{r}X_1 \qquad \text{and} \qquad  A\left(X_2\right)=- \frac{r}{\sqrt{1+r^2}}X_2,
	$$
	where $X_1$ is a vector field tangent to the $k$-dimensional Euclidean sphere of radius $r$ and $X_2$ is a vector field tangent to the $(m-k)$-dimensional hyperbolic space with sectional curvature equal to $-1/\left(1+r^2\right)$.
	
	We have
	$$
	H=-\frac{k+r^2m}{mr\sqrt{1+r^2}}\eta, \quad \Ric\left(X_1\right)=\frac{k-1}{r^2}X_1, \quad \Ric\left(X_2\right)=-\frac{m-k-1}{1+r^2}X_2
	$$
	and
	$$
	\Scal=\frac{k(k-1)}{r^2}-\frac{(m-k)(m-k-1)}{1+r^2}.
	$$
\end{itemize}

Further, we compute $\tau_2^c(\iota)$ for the hypersurfaces $M^m$ of $\mathbb{H}^{m+1}$ described above.

First, we consider the canonical inclusion 
$$
\iota: M^m=\left\{\bar{x}\in\mathbb{H}^{m+1}\ | \ x^{1}=r\geq 0\right\}\to\mathbb{H}^{m+1}
$$ 
and by a direct computation we obtain
$$
\tau(\iota)=-\frac{mr}{\sqrt{1+r^2}}\eta, \qquad \tau_2(\iota)=\frac{m^2r\left(1+2r^2\right)}{\left(1+r^2\right)^{3/2}}\eta
$$
such that
$$
\tau_2^c(\iota)=\frac{mr\left(6mr^2-2m^2+11m-6\right)}{3\left(1+r^2\right)^{3/2}}\eta.
$$

Thus, we can state the following result which agrees with Theorem \ref{c-biharmonic-umbilical}.

\begin{Theorem}\label{prop-c-biharmonic-equidistant}
The hyperbolic hypersurface $M^m=\left\{\bar{x}\in\mathbb{H}^{m+1}\ | \ x^{1}=r\geq 0\right\}$ is c-biharmonic in $\mathbb{H}^{m+1}$ if and only if 
\begin{itemize}
	\item[i)]  $r=0$, i.e., $M^m$ is totally geodesic in $\mathbb{H}^{m+1}$;
	\item[ii)] $m\geq 5$ and
	$$
	r^2=\frac{2m^2-11m+6}{6m}.
	$$ 
\end{itemize} 
\end{Theorem}

Second, we consider the canonical inclusion
$$
\iota: M^m=\left\{\bar{x}\in\mathbb{H}^{m+1}\ | \ x^{m+2}=x^{m+1}+a,\  a>0\right\}\to \mathbb{H}^{m+1},
$$
and we have
$$
\tau(\iota)= -m\eta, \qquad \tau_2(\iota)=2m^2\eta,\qquad \tau_2^c(\iota)=2m^2\eta\neq 0, \quad  m\geq 2.
$$
Therefore, we get
\begin{Prop}
The flat hypersurface $M^m=\left\{\bar{x}\in\mathbb{H}^{m+1}\ | \ x^{m+2}=x^{m+1}+a,\  a>0\right\}$ cannot be c-biharmonic in $\mathbb{H}^{m+1}$.
\end{Prop}

For the third case, if we consider the canonical inclusion
$$
\iota: M^m=\left\{\bar{x}\in\mathbb{H}^{m+1}\ | \ \sum_{i=1}^{m+1}\left(x^{i}\right)^2=r^2,\  r>0\right\}\to \mathbb{H}^{m+1},
$$
we have
$$
\tau(\iota)= -\frac{m\sqrt{1+r^2}}{r}\eta, \qquad \tau_2(\iota)=\frac{m^2\sqrt{1+r^2}\left(1+2r^2\right)}{r^3}\eta,
$$
and
$$
\tau_2^c(\iota)=\frac{m\sqrt{1+r^2}\left(6mr^2+2m^2-5m+6\right)}{3r^3}\eta\neq 0, \quad m\geq 2.
$$
Thus, it follows
\begin{Prop}
The spherical hypersurface $M^m=\left\{\bar{x}\in\mathbb{H}^{m+1}\ | \ \sum_{i=1}^{m+1}\left(x^{i}\right)^2=r^2,\  r>0\right\}$ cannot be c-biharmonic in $\mathbb{H}^{m+1}$.
\end{Prop}

For the last case, if we consider the canonical inclusion
\begin{align*}
\iota: M^m_k=\left\{\left(\bar{x}_1,\bar{x}_2\right)\in\mathbb{H}^{m+1}\ \right.|&\left. \ \left|\bar{x}_1\right|^2= \sum_{i=1}^{k+1}\left(x^{i}\right)^2=r^2,\ \left|\bar{x}_2\right|^2=\sum_{i=k+2}^{m+1}\left(x^{i}\right)^2 - \left(x^{m+2}\right)^2=-\left(1+r^2\right),\right. \\
&\left.  r>0, 0\leq k\leq m-1\right\}\to\mathbb{H}^{m+1},
\end{align*}
we have
$$
\tau(\iota)=-\frac{k+r^2m}{r\sqrt{1+r^2}}\eta, \qquad \tau_2(\iota)=\frac{k+r^2m}{r\sqrt{1+r^2}}\left(\frac{k\left(1+r^2\right)}{r^2}+\frac{(m-k)r^2}{1+r^2}+m\right)\eta
$$
and
\begin{align}\label{eq:tau2c-case4}
\tau_2^c(\iota) =  \frac{1}{3r^3\left(1+r^2\right)^{3/2}} & \left(  6m^2r^6+\left(-2m^3+11m^2-6m+4k\left(m^2-m+3\right)\right)r^4 \right.\nonumber\\
 &\quad + \left. 2k\left(k\left(3m-5\right)-m^2+3m+6\right)r^2+2k^3-5k^2+6k\right) \eta.
\end{align}
We note that when $k=0$, we reobtain Theorem \ref{prop-c-biharmonic-equidistant}. Thus, we assume $k>0$.

From \eqref{eq:tau2c-case4} it is easy to state the following characterization result.

\begin{Prop}\label{characterization-hyperbolic}
The hypersurface 
\begin{align*}
	M^m_k=\left\{\left(\bar{x}_1,\bar{x}_2\right)\in\mathbb{H}^{m+1}\ \right.|&\left. \ \left|\bar{x}_1\right|^2= \sum_{i=1}^{k+1}\left(x^{i}\right)^2=r^2,\ \left|\bar{x}_2\right|^2=\sum_{i=k+2}^{m+1}\left(x^{i}\right)^2 - \left(x^{m+2}\right)^2=-\left(1+r^2\right),\right. \\
	&\left.  r>0,0<k\leq m-1\right\}
\end{align*}
is c-biharmonic if and only if $r^2=T_\ast$ is a positive solution of the polynomial equation 
\begin{equation}\label{polynomial-equation-hyperbolic}
\zeta(T)=a_3T^3+a_2T^2+a_1T+a_0=0,
\end{equation}
with the coefficients given by
\begin{equation*}
	\left\{
	\begin{array}{l}
		a_0=2k^3-5k^2+6k\\\\
		a_1=2k\left(k\left(3m-5\right)-m^2+3m+6\right)\\\\
		a_2=-2m^3+11m^2-6m+4k\left(m^2-m+3\right)\\\\
		a_3=6m^2
	\end{array}
	\right..
\end{equation*}
\end{Prop}

\begin{Prop}\label{non-existence}
The hypersurface 
\begin{align*}
M^m_k=\left\{\left(\bar{x}_1,\bar{x}_2\right)\in\mathbb{H}^{m+1}\ \right.|&\left. \ \left|\bar{x}_1\right|^2= \sum_{i=1}^{k+1}\left(x^{i}\right)^2=r^2,\ \left|\bar{x}_2\right|^2=\sum_{i=k+2}^{m+1}\left(x^{i}\right)^2 - \left(x^{m+2}\right)^2=-\left(1+r^2\right),\right. \\
&\left.  r>0,0<k\leq m-1\right\}
\end{align*}
cannot be c-biharmonic in $\mathbb{H}^{m+1}$ if
\begin{itemize}
	\item[i)] $m\in\left\{2,3,10,11, 12,...\right\}$ and 
	$$
	 \frac{m\left(m^2-11m+6\right)}{4\left(m^2-m+3\right)}\leq k;
	$$
	\item[ii)] $m\in\left\{4,5,6,7,8,9\right\}$ and
	$$
	\frac{m^2-3m-6}{3m-5} \leq k.
	$$
\end{itemize}

\end{Prop}

\begin{proof}

From Proposition \ref{characterization-hyperbolic} we have that $M^m_k$ is c-biharmonic in $\mathbb{H}^{m+1}$ if and only if the polynomial equation \eqref{polynomial-equation-hyperbolic} has positive solutions. We note that $a_0>0$ and $a_3>0$, for any $m\geq 2$ and $0<k\leq m-1$. If $a_1\geq 0$ and $a_2\geq 0$, it is clear that \eqref{polynomial-equation-hyperbolic} does not have positive solutions. Now, by some standard computations we conclude.
\end{proof}

\begin{Bem}
Using a computer algebra system, one can see that for $m$ large enough, 
$$
	\frac{m\left(m^2-11m+6\right)}{4\left(m^2-m+3\right)}\in \left(\frac{m}{3}, \frac{m}{2}\right).
$$
\end{Bem}

When $m\leq 7$, we have the following non-existence result.

\begin{Theorem}
The hypersurface 
\begin{align*}
	M^m_k=\left\{\left(\bar{x}_1,\bar{x}_2\right)\in\mathbb{H}^{m+1}\ \right.|&\left. \ \left|\bar{x}_1\right|^2= \sum_{i=1}^{k+1}\left(x^{i}\right)^2=r^2,\ \left|\bar{x}_2\right|^2=\sum_{i=k+2}^{m+1}\left(x^{i}\right)^2 - \left(x^{m+2}\right)^2=-\left(1+r^2\right),\right. \\
	&\left. r>0,0<k\leq m-1\right\}
\end{align*}
cannot be c-biharmonic in $\mathbb{H}^{m+1}$ if $m\leq 7$. 
\end{Theorem}
\begin{proof}
For $m\leq 6$, and for any $k\leq m-1$, the conclusion follows directly from Proposition \ref{non-existence}. Using a computer algebra system, we can check that for $m=7$, also the hypersurface $M^7_k$ cannot be c-biharmonic.
\end{proof}

On the contrary, for $m\geq 8$, we do have examples of c-biharmonic hypersurfaces $M^m_k$ in $\mathbb{H}^{m+1}$.

\begin{Theorem}
If $k=1$, then the hypersurface 
\begin{align*}
	M^m_{1}=\left\{\left(\bar{x}_1,\bar{x}_2\right)\in\mathbb{H}^{m+1}\ \right.|&\left. \ \left|\bar{x}_1\right|^2= \sum_{i=1}^{2}\left(x^{i}\right)^2=r^2,\ \left|\bar{x}_2\right|^2=\sum_{i=3}^{m+1}\left(x^{i}\right)^2 - \left(x^{m+2}\right)^2=-\left(1+r^2\right),\right. \\
    &\left. r>0 \right\}
\end{align*}
is c-biharmonic in $\mathbb{H}^{m+1}$ if and only if $m\geq 8$. In this case, $r^2$ is one of the two positive solutions of the polynomial equation \eqref{polynomial-equation-hyperbolic}. Moreover, for $m=8$ and $2\leq k$, $M^m_k$ cannot be c-biharmonic in $\mathbb{H}^{9}$.
\end{Theorem}

\begin{proof}
Since $a_0>0$ and $a_3>0$, from Descartes' rule of signs, we see that the polynomial equation \eqref{polynomial-equation-hyperbolic} has exactly $2$ or $0$ positive solutions. 

We can check that $\zeta\left(1/m\right)<0$, for any $m\geq 8$, and as $\zeta(0)=a_0>0$ and $\zeta(T)$ goes to infinity when $T$ goes to infinity, \eqref{polynomial-equation-hyperbolic} has two positive solutions, one in $\left(0,1/m\right)$ and the other one in $\left(1/m,\infty\right)$. 

For the last part of the proof, one can verify that, when $m=8$, the lower bound for $k$ given in Proposition \ref{non-existence} satisfies
$$
\frac{m^2-3m-6}{3m-5}\in (1,2),
$$
so we conclude.
\end{proof}

We can enlarge the range for $k$ such that $M^m_k$ becomes c-biharmonic in $\mathbb{H}^{m+1}$.

\begin{Prop}
If $m\geq 8$ and $1\leq k<k_m$, then $M^m_k$ is c-biharmonic in $\mathbb{H}^{m+1}$ if and only if $r^2$ is one of the two positive solutions of the polynomial equation \eqref{polynomial-equation-hyperbolic}, where
$$
k_m=\frac{2m^3-6m^2-3m+m\sqrt{4m^4-24m^3-48m^2+84m-63}}{4\left(3m^2-2m+3\right)}.
$$
\end{Prop}

\begin{proof}
As in the previous proof, we can check that
$$
\zeta\left(\frac{k}{m}\right)= \frac{2k\left(2k^2\left(3m^2-2m+3\right)+k\left(-2m^3+6m^2+3m\right)+3m^2\right)}{m^2}
$$
is negative for any $1\leq k<k_m$, and for any $m\geq 8$. Thus, we conclude.
\end{proof}

\begin{Bem}
Using a computer algebra system, one can see that for $m$ large enough, 
$$
k_m\in \left(\frac{m}{4},\frac{m}{3}\right).
$$
\end{Bem}

\begin{Bem}
The case $m=1$ does not represent a singularity for the c-biharmonic equation, but it provides no solution; this observation agrees with Proposition \ref{prop-energies-equal}, i).
\end{Bem}

\section{The stability of c-biharmonic hyperspheres}

Since we are going to compute the index and nullity for  the c-biharmonic maps given in Theorem \ref{thm:inclusion-sphere}, we consider from the beginning only the case $M^m=\mathbb{S}^{m}(r)$ and $N^n=\mathbb{S}^{n}$.

Let $\phi:\mathbb{S}^{m}(r)\to\mathbb{S}^{n}$ be a smooth map. We consider $\left\{\phi_{s,t}\right\}_{s,t\in\mathbb{R}}$ a two parameter variation of $\phi$, i.e., we consider a smooth map $\Phi:\mathbb{R}\times\mathbb{R}\times \mathbb{S}^{m}(r)\to \mathbb{S}^{n}$, $\Phi(s,t,p)=\phi_{s,t}(p)$, such that 
$$
\Phi(0,0,p)=\phi_{0,0}(p)=\phi(p),\qquad \forall p\in \mathbb{S}^{m}(r).
$$
We set $V$, $W\in C\left(\phi^{-1}T\mathbb{S}^n\right)$ as
$$
V(p)=\left.\frac{d}{ds}\right|_{s=0}\left\{\phi_{s,0}(p)\right\}=d\Phi_{(0,0,p)}\left(\frac{\partial}{\partial s}\right)\in T_{\phi(p)}\mathbb{S}^n
$$
and
$$
W(p)=\left.\frac{d}{dt}\right|_{t=0}\left\{\phi_{0,t}(p)\right\}=d\Phi_{(0,0,p)}\left(\frac{\partial}{\partial t}\right)\in T_{\phi(p)}\mathbb{S}^n.
$$
The following formulas for the second variation of the energy and bienergy functionals are well-known (see \cite{MR886529} for the bienergy, see \cite[Chapter 5]{MR1252178} for the Jacobi operator associated with harmonic maps)
\begin{equation*}\label{energy-functional-2-derivative*}
\left.\frac{\partial^2}{\partial s\partial t}\right|_{(s,t)=(0,0)}\left\{E\left(\phi_{s,t}\right)\right\} =\int_{\mathbb{S}^{m}(r)} \left(\langle W, J(V)\rangle-\langle \nabla^{\Phi}_{\frac{\partial}{\partial s}}d\Phi\left(\frac{\partial}{\partial t}\right),\tau(\phi)\rangle\right)\ v_g,		
\end{equation*}
where 
\begin{equation*}\label{eq:J}
J(V)=\bar\Delta V+\tr R^N\left(d\phi(\cdot),V\right)d\phi(\cdot),
\end{equation*}
and
\begin{equation*}\label{bienergy-functional-2-derivative*}
\left.\frac{\partial^2}{\partial s\partial t}\right|_{(s,t)=(0,0)}\left\{E_2\left(\phi_{s,t}\right)\right\} =\int_{\mathbb{S}^{m}(r)} \left(\langle W, J_2(V)\rangle+\langle \nabla^{\Phi}_{\frac{\partial}{\partial s}}d\Phi\left(\frac{\partial}{\partial t}\right),\tau_2(\phi)\rangle\right)\ v_g,		
\end{equation*}
where 
\begin{align*}\label{eq:J2}
J_2(V) = &\bar\Delta\bar\Delta V+\bar\Delta\left(\tr \langle V,d\phi(\cdot)\rangle d\phi(\cdot)-|d\phi|^2 V \right)+2\langle d\tau(\phi),d\phi\rangle V+|\tau(\phi)|^2V \\
& - 2\tr \langle V,d\tau(\phi)(\cdot)\rangle d\phi(\cdot)-2\tr\langle \tau(\phi),dV(\cdot)\rangle d\phi(\cdot)-\langle \tau(\phi),V\rangle \tau(\phi)+\tr\langle d\phi(\cdot),\bar{\Delta}V\rangle d\phi(\cdot)\\
& + \tr\langle d\phi(\cdot),\left(\tr\langle V,d\phi(\cdot)\rangle d\phi(\cdot)\right)\rangle d\phi(\cdot)-2|d\phi|^2\tr\langle d\phi(\cdot),V\rangle d\phi(\cdot)+2\langle dV,d\phi\rangle \tau(\phi)\\
& - |d\phi|^2\bar{\Delta}V+|d\phi|^4V,
\end{align*}
see for example \cite{MR1943720} for a detailed derivation.
Since in our case
$$
E^c_2\left(\phi_{s,t}\right)=\frac{1}{2}\int_{\mathbb{S}^{m}(r)}\left(\left|\tau\left(\phi_{s,t}\right)\right|^2+\frac{2}{3}\frac{(m-1)(m-3)}{r^2}\left|d\phi_{s,t}\right|^2\right) \ v_g,
$$
if the map $\phi$ is c-biharmonic we infer that
\begin{equation*}\label{c-bienergy-functional-2-derivative*}
\left.\frac{\partial^2}{\partial s\partial t}\right|_{(s,t)=(0,0)}\left\{E^c_2\left(\phi_{s,t}\right)\right\} =\int_{\mathbb{S}^{m}(r)} \langle W, J_2^c(V)\rangle \ v_g =: \left(W, J_2^c(V)\right),		
\end{equation*}
where
\begin{equation*}\label{eq:J2c}
	J_2^c(V)=J_2(V)+\frac{2}{3}\frac{(m-1)(m-3)}{r^2}J(V).
\end{equation*}
For $m=1$ or $m=3$, we note that the Hessians corresponding to $E_2^c$ and $E_2$ coincide. From the biharmonic map theory it is known that the index of the biharmonic hypersphere $\mathbb{S}^m\left(1/\sqrt{2}\right)$ in $\mathbb{S}^{m+1}$ is equal to one, for any $m\geq1$ (see \cite{MR2135286} and \cite{MR2187367}), and there is no stable compact proper biharmonic submanifold in $\mathbb{S}^n$. As we will see next, we do have examples of stable compact c-biharmonic submanifolds in $\mathbb{S}^n$. 

\subsection{The stability of the equator of $\mathbb{S}^{m+1}$}
Let us consider $\iota:\mathbb{S}^m\to\mathbb{S}^{m+1}$ the canonical inclusion of the equator of $\mathbb{S}^{m+1}$. We can choose the unit section in the normal bundle  $\eta=\bar{e}_{m+2}$ and obviously $A=A_\eta=0$. We have the decomposition
\begin{equation*}
C\left(\iota^{-1}T\mathbb{S}^{m+1}\right)  =  \left\{\alpha\eta \ |\ \alpha\in C^{\infty}\left(\mathbb{S}^m\right)\right\} \oplus  \left\{V\ |\ V\in C\left(T\mathbb{S}^m\right)\right\}\\
\end{equation*} 
and
\begin{align*}
\oplus_{j=0}^{\infty}\left\{\alpha\eta \ |\ \Delta\alpha=\lambda_j\alpha\right\} \oplus\oplus_{\mathcal{l}=0}^{\infty} \left\{V\in C\left(T\mathbb{S}^m\right) |\ \Delta_H(V)=\mu_{\mathcal{l}} V\right\}
\end{align*}
is dense in the space of the $L^2$-sections of $\phi^{-1}T\mathbb{S}^{m+1}$, where $\Delta_H$ is the Hodge Laplacian acting on $p$-forms of the domain. In particular, identifying $1$-forms with vector fields, we have
$$
\Delta_H(X)=-\tr\nabla^2 X+\Ric(X).
$$

First, we consider the case $V=\alpha\eta$, where $\Delta\alpha=\lambda\alpha$. By standard computations, we obtain
$$
J(V)=(\lambda-m)V, \qquad J_2(V)=(\lambda-m)^2 V
$$
and thus
$$
J^c_2(V)=(\lambda-m)\left(\lambda+\frac{2m^2-11m+6}{3}\right)V.
$$

Second, we consider $V\in C\left(T\mathbb{S}^m\right)$ with $\Delta_H(V)=\mu V$ and we get
$$
J(V)=\left(\mu-2m+2\right)V,\qquad J_2(V)= \left(\mu-2m+2\right)^2V
$$
and 
$$
J_2^c(V)=\left(\mu-2m+2\right)\left(\mu+\frac{2m^2-14m+12}{3}\right)V.
$$
We see that $J_2^c$ preserves the subspaces $C\left(N\mathbb{S}^m\right)$ and $C\left(T\mathbb{S}^m\right)$. Consequently, we compute the index and nullity of $\left.J_{2}^c\right|_{C\left(N\mathbb{S}^m\right)}$ and $\left.J_{2}^c\right|_{C\left(T\mathbb{S}^m\right)}$, respectively and then we sum them up. 

Recall that the eigenvalues of the Laplacian acting on $0$-forms of $\mathbb{S}^m$ are given by
$$
\lambda=\lambda_j=j(m+j-1), \qquad j\geq 0.
$$
Looking at the normal part and denoting 
$$
\gamma=(\lambda-m)\left(\lambda+\frac{2m^2-11m+6}{3}\right),
$$
we see that if $m\in\left\{1,2,3,4\right\}$, then $\gamma$ is positive for $j=0$ and any $j\geq 2$ and vanishes when $j=1$. If $m\geq 5$, then $\gamma$ is negative for $j=0$, positive for any $j\geq 2$ and vanishes for $j=1$. 

We conclude that, when $m\in\left\{1,2,3,4\right\}$, the index of $\left.J_{2}^c\right|_{C\left(N\mathbb{S}^m\right)}$ is $0$. When $m\geq 5$, the index of $\left.J_{2}^c\right|_{C\left(N\mathbb{S}^m\right)}$ is $1$ and $\left.J_{2}^c\right|_{C\left(N\mathbb{S}^m\right)}$ is negative definite on $\left\{a\eta \ |\ a\in \mathbb{R}\right\}$.

The nullity of $\left.J_{2}^c\right|_{C\left(N\mathbb{S}^m\right)}$ is $m+1$, for any $m\geq 1$, and $\left.J_{2}^c\right|_{C\left(N\mathbb{S}^m\right)}$ vanishes on $\left\{\alpha\eta \ |\ \Delta\alpha=\lambda_1\alpha \right\}$.

Concerning the tangent part, it is known that the eigenvalues of $\Delta_H$ are
$$
\left\{\mu_\mathcal{l}\right\}=\left\{ \lambda_j \ |\ j\geq 1 \right\}\cup \left\{ (k+1)(k+m-2) \ |\ k\geq 1 \right\}
$$
(see \cite{MR2531148, MR510492, MR533709}). 

The first set of eigenvalues corresponds to the restriction of $\Delta_H$ to the vector fields tangent to $\mathbb{S}^m$ of type $\nabla\alpha$, $\alpha$ being a smooth function on $\mathbb{S}^m$, and the second set corresponds to the restriction of $\Delta_H$ to the divergence free vector fields tangent to 
$\mathbb{S}^m$. When $m=2$, the two sets coincide. Moreover, we recall that if $\Div X=0$, then $\Delta_HX=2(m-1)X$ and $\Div X=0$ if and only if $X$ is Killing.

By direct checking, we obtain that, when $m\in\left\{1,2,3,4\right\}$, the index of $\left.J_{2}^c\right|_{C\left(T\mathbb{S}^m\right)}$ is $0$. When $m\geq 5$, the index of $\left.J_{2}^c\right|_{C\left(T\mathbb{S}^m\right)}$ is $m+1$ and $\left.J_{2}^c\right|_{C\left(T\mathbb{S}^m\right)}$ is negative definite on $\left\{\nabla\alpha \ |\ \Delta\alpha =\lambda_1\alpha \right\}$.

The nullity of $\left.J_{2}^c\right|_{C\left(T\mathbb{S}^m\right)}$ is $m(m+1)/2$, for any positive integer $m$ with the exception of $m=2$ and $m=4$. For $m\neq 2$ and $m\neq4$, the operator $\left.J_{2}^c\right|_{C\left(T\mathbb{S}^m\right)}$ vanishes on the set of all Killing vector fields of $\mathbb{S}^m$. When $m=2$ or $m=4$, the nullity of $\left.J_{2}^c\right|_{C\left(T\mathbb{S}^m\right)}$ is $m(m+1)/2 + (m+1)$ and $\left.J_{2}^c\right|_{C\left(T\mathbb{S}^m\right)}$ vanishes on the set of all Killing vector fields of $\mathbb{S}^m$ and on $\left\{\nabla \alpha \ |\ \Delta\alpha=\lambda_1\alpha \right\}$.

We can summarize the above results and we state

\begin{Theorem}\label{thm:stability-equator}
 We consider $\iota:\mathbb{S}^m\to\mathbb{S}^{m+1}$ the canonical inclusion of the totally geodesic hypersphere $\mathbb{S}^m$, thought of as a c-biharmonic map. Then, we have
  \begin{itemize}
	\item [i)] if $m\in\{1,2,3,4\}$, then the index of $\mathbb{S}^m$ is $0$, i.e., $\mathbb{S}^m$ is stable;
    \item[ii)] if $m\geq 5$, then the index of $\mathbb{S}^m$ is $m+2$.
  \end{itemize}
\end{Theorem}

\begin{Theorem}
  We consider $\iota:\mathbb{S}^m\to\mathbb{S}^{m+1}$ the canonical inclusion of the totally geodesic hypersphere $\mathbb{S}^m$, thought of as a c-biharmonic map. Then, we have
  \begin{itemize}
	\item [i)] if $m\neq 2$ and $m\neq 4$, then the nullity of $\mathbb{S}^m$ is $(m+1)(m+2)/2$;
    \item[ii)] if $m=2$ or $m=4$, then the nullity of $\mathbb{S}^m$ is $(m+1)(m+4)/2$.
  \end{itemize}
\end{Theorem}
From the above theorems it is not difficult to obtain the index and the nullity when we consider the identity map of $\mathbb{S}^m$. More precisely, we have

\begin{Cor}\label{stability-identity}
We consider $\Id:\mathbb{S}^m\to\mathbb{S}^{m}$ the identity map of $\mathbb{S}^m$, thought of as a c-biharmonic map. Then, we have
\begin{itemize}
	\item[i)] if $m=1$ or $m=3$, then the index and the nullity of $\mathbb{S}^m$ are $0$ and $m(m+1)/2$, respectively;
	\item[ii)] if $m=2$ or $m=4$, then the index and the nullity of $\mathbb{S}^m$ are $0$ and $(m+1)(m+2)/2$, respectively;
	\item[iii)] if $m\geq 5$, then the index and the nullity of $\mathbb{S}^m$ are $m+1$ and $m(m+1)/2$, respectively.
\end{itemize}
\end{Cor}

\subsection{The stability of the small hyperspheres of $\mathbb{S}^{m+1}$}
Let us consider $\iota:\mathbb{S}^m(r)\to\mathbb{S}^{m+1}$ the canonical inclusion of a small hypersphere of radius $r\in (0,1)$ in $\mathbb{S}^{m+1}$. We can choose the unit section in the normal bundle  $\eta=\left(\left(\sqrt{1-r^2}/r\right)\bar{x}, -r\right)$ and thus $A=A_\eta=-\left(\sqrt{1-r^2}/r\right)\Id$. We have the decomposition
\begin{align*}
	C\left(\iota^{-1}T\mathbb{S}^{m+1}\right) & =  \left\{\alpha\eta \ |\ \alpha\in C^{\infty}\left(\mathbb{S}^m(r)\right)\right\} \oplus  \left\{V\ |\ V\in C\left(T\mathbb{S}^m(r)\right)\right\}\\
	& = \left\{\alpha\eta \ |\ \alpha\in C^{\infty}\left(\mathbb{S}^m(r)\right)\right\} \oplus  \left\{\nabla \alpha\ |\ \alpha\in C^{\infty}\left(\mathbb{S}^m(r)\right)\right\}\\
	& \qquad \oplus \left\{V\in C\left(T\mathbb{S}^m(r)\right) \ |\ \Div V=0\right\},
\end{align*} 
and
\begin{align*}
	\oplus_{j=0}^{\infty}\left\{\alpha\eta \ |\ \Delta\alpha  = \lambda_j\alpha\right\} &\oplus\oplus_{j=1}^{\infty} \left\{\nabla\alpha \ |\ \Delta\alpha=\lambda_{j} \alpha\right\}\\
	& \oplus \oplus_{k=1}^\infty\left\{V \in C\left(T\mathbb{S}^m(r)\right) \ |\ \Div V=0, \Delta_H V=\mu_k V\right\}
\end{align*}
is dense in the space of the $L^2$-sections of $\iota^{-1}T\mathbb{S}^{m+1}$. Recall that 
$$
\lambda_j=\frac{j(m+j-1)}{r^2}, \ j\geq 0\qquad \text{and} \qquad \mu_k=\frac{(k+1)(k+m-2)}{r^2},\ k\geq 1.
$$

First, we consider the case $V=\alpha\eta$, where $\Delta\alpha=\lambda\alpha$. By standard computations, we obtain
$$
J(V)=\left(\frac{m\left(1-r^2\right)}{r^2}+\lambda-m\right)V-\frac{2\sqrt{1-r^2}}{r}\nabla\alpha, 
$$
\begin{align*}
J_2(V) = & \left(\left(\frac{m\left(1-r^2\right)}{r^2}+\lambda\right)^2+\frac{2\left(2\lambda-3m^2\right)\left(1-r^2\right)}{r^2}-2m\lambda+m^2\right)V \\
& +\frac{2\sqrt{1-r^2}}{r}\left(\frac{(m+1)r^2-2}{r^2}-2\lambda+3m-1\right)\nabla\alpha
\end{align*}
and thus
\begin{align*}
J_2^c(V) = & \left(\left(\frac{m\left(1-r^2\right)}{r^2}+\lambda\right)^2+\frac{2\left(2\lambda-3m^2\right)\left(1-r^2\right)}{r^2}-2m\lambda+m^2\right.\\
&\quad\left.+\frac{2}{3}\frac{(m-1)(m-3)}{r^2}\left(\frac{m\left(1-r^2\right)}{r^2}+\lambda-m\right) \vphantom{\left(\frac{m\left(1-r^2\right)}{r^2}+\lambda\right)^2}\right)V \\
	&+\frac{2\sqrt{1-r^2}}{r}\left(\frac{(m+1)r^2-2}{r^2}-2\lambda+3m-1-\frac{2}{3}\frac{(m-1)(m-3)}{r^2}\right)\nabla\alpha.
\end{align*}

If we consider $V=a\eta$, $a\in \mathbb{R}$, we get
\begin{align*}
	J_2^c(V)  = & \left(\left(\frac{m\left(1-r^2\right)}{r^2}\right)^2-\frac{6m^2\left(1-r^2\right)}{r^2}+m^2\right.\\
	&\quad \left.+\frac{2}{3}\frac{(m-1)(m-3)}{r^2}\left(\frac{m\left(1-r^2\right)}{r^2}-m\right)\vphantom{\left(\frac{m\left(1-r^2\right)}{r^2}\right)^2}\right)V \\
	 = & \frac{m}{r^4}\left(8mr^4
	-4\left(\frac{m^2}{3}+\frac{2m}{3}+1\right)r^2
	+\frac{2m^2}{3}-\frac{5m}{3}+2\right)V.
\end{align*}

Now, we can see that, if $m=1$ or $m=3$ and $r=1/\sqrt{2}$, or $m=2$ and $r=1/\sqrt{3}$, or $m=4$ and $r=\sqrt{3}/2$, we have $\left(J_2^c(V), V\right)<0$. Therefore the index of $\mathbb{S}^m(r)$ is at least $1$. We note that, as in the standard biharmonic case, the index is produced by the sections normal to the hypersphere.

Second, we consider the case $V=\nabla\alpha$, where $\alpha\in C^\infty\left(\mathbb{S}^m(r)\right)$ satisfies $\Delta\alpha=\lambda\alpha$, and we infer

$$
J(V)=\left(\frac{2-r^2-m}{r^2}+\lambda-m+1\right)V-\frac{2\sqrt{1-r^2}}{r}\lambda\alpha\eta,
$$
\begin{align*}
J_2(V)=&\left(\left(\frac{2-r^2-m}{r^2}+\lambda\right)^2+(1-m)(1-m+2\lambda)\right.\\
&\quad \left.+\frac{\left(1-r^2\right)\left(4\lambda-m^2\right)-\left(2-r^2-m\right)(2m-2)}{r^2}\vphantom{\left(\frac{2-r^2-m}{r^2}+\lambda\right)^2}\right)V\\
&+\frac{2\sqrt{1-r^2}}{r}\left(-\frac{2-r^2}{r^2}-2\lambda+4m-1\right)\lambda\alpha\eta
\end{align*}
and therefore
\begin{align*}
	J_2^c(V)=&\left(\left(\frac{2-r^2-m}{r^2}+\lambda\right)^2+(1-m)(1-m+2\lambda)\right.\\
	&\quad +\frac{\left(1-r^2\right)\left(4\lambda-m^2\right)-\left(2-r^2-m\right)(2m-2)}{r^2}\\
	&\quad \left. +\frac{2}{3}\frac{(m-1)(m-3)}{r^2}\left(\frac{2-r^2-m}{r^2}+\lambda-m+1\right)\vphantom{\left(\frac{2-r^2-m}{r^2}+\lambda\right)^2}\right)V\\
	&+\frac{2\sqrt{1-r^2}}{r}\left(-\frac{2-r^2}{r^2}-2\lambda+4m-1-\frac{2}{3}\frac{(m-1)(m-3)}{r^2}\right)\lambda\alpha\eta.
\end{align*}

In the last case, we assume that $V\in C\left(T\mathbb{S}^m(r)\right)$ satisfies $\Div V=0$ and $\Delta_H V=\mu V$. By straightforward computations, we obtain
$$
J(V)=\left(\frac{2-r^2-m}{r^2}+\mu+1-m\right)V,
$$
$$
J_2(V)=\left(\left(\frac{2-r^2-m}{r^2}+\mu\right)^2+2(1-m)\left(\frac{2-r^2-m}{r^2}+\mu\right)-\frac{m^2\left(1-r^2\right)}{r^2}+(1-m)^2\right)V
$$
and 
\begin{align}\label{J2c-divV}
J_2^c(V)= &\left(\left(\frac{2-r^2-m}{r^2}+\mu\right)^2+2(1-m)\left(\frac{2-r^2-m}{r^2}+\mu\right)-\frac{m^2\left(1-r^2\right)}{r^2}+(1-m)^2\right. \nonumber\\
&\quad \left. +\frac{2}{3}\frac{(m-1)(m-3)}{r^2}\left(\frac{2-r^2-m}{r^2}+\mu+1-m\right)\vphantom{\left(\frac{2-r^2-m}{r^2}+\mu\right)^2}\right)V.
\end{align}

We denote by  $S_0=\left\{a\eta\ |\ a\in\mathbb{R}\right\}$, $S_j=\left\{\alpha\eta \ |\ \Delta \alpha=\lambda_j\alpha\right\}\oplus \left\{\nabla\alpha \ |\ \Delta\alpha=\lambda_j\alpha\right\}$ and $\tilde{S}_k=\left\{V\in C\left(T\mathbb{S}^m(r)\right)\ |\  \Div V=0, \Delta_H V=\mu_k V\right\}$. We note that
$$
J_2^c\left(S_0\right)\subset S_0, \qquad J_2^c\left(S_j\right)\subset S_j, \ j\geq 1, \qquad J_2^c\left(\tilde{S}_k\right)\subset \tilde{S}_k, \ k\geq 1
$$
and these spaces are mutually orthogonal.

Let $\left\{\alpha_{j,1}, \dots, \alpha_{j,s}\right\}$ be an $L^2$-orthonormal basis in $\left\{\alpha\in C^\infty\left(\mathbb{S}^m(r)\right)\ |\ \Delta\alpha=\lambda_j\alpha\right\}$, $j\geq 1$. Then,
$$
\mathcal{B}_j=\left\{\alpha_{j,1}\eta, \dots, \alpha_{j,s}\eta, \frac{1}{\sqrt{\lambda_j}}\nabla\alpha_{j,1},\dots, \frac{1}{\sqrt{\lambda_j}}\nabla\alpha_{j,s} \right\}
$$
is an $L^2$-orthonormal basis in $S_j$ with respect to $\mathcal{B}_j$, $J_2^c$ restricted to $S_j$ has the following matrix representation

\begin{equation}\label{matrix}	
\left[
\begin{array}{c|c}
	\begin{array}{c c c c c c}
a & 0 & 0  & \dots &  0 & 0 \\
0 & a & 0 & \dots &  0 & 0 \\
0 & 0 & a & \dots &  0 & 0 \\
\vdots & \vdots & \vdots & \dots &  \vdots & \vdots \\ 
0 & 0 & 0 & \dots &  a & 0 \\
0 & 0 & 0 & \dots &  0 & a 
\end{array} & \begin{array}{c c c c c c}
d & 0 & 0 & \dots & 0 & 0  \\
 0 & d & 0 & \dots & 0 & 0  \\
 0 & 0 & d & \dots & 0 & 0  \\
\vdots & \vdots & \vdots & \dots & \vdots & \vdots\\ 
 0 & 0 & 0 & \dots & d & 0  \\
 0 & 0 & 0 & \dots & 0 & d 
\end{array}
\\ \hline
\begin{array}{c c c c c c}
d & 0 & 0  & \dots &  0 & 0 \\
0 & d & 0 & \dots &  0 & 0 \\
0 & 0 & d & \dots &  0 & 0 \\
\vdots & \vdots & \vdots & \dots &  \vdots & \vdots \\ 
0 & 0 & 0 & \dots &  d & 0 \\
0 & 0 & 0 & \dots &  0 & d 
\end{array} & \begin{array}{c c c c c c}
b & 0 & 0 & \dots & 0 & 0  \\
0 & b & 0 & \dots & 0 & 0  \\
0 & 0 & b & \dots & 0 & 0  \\
\vdots & \vdots & \vdots & \dots & \vdots & \vdots\\ 
0 & 0 & 0 & \dots & b & 0  \\
0 & 0 & 0 & \dots & 0 & b 
\end{array}
\end{array}
\right]
\end{equation}

where 

\begin{align*}
a=a_j = & \left(J_2^c\left(\alpha_{j,\mathcal{l}}\eta\right),\alpha_{j,\mathcal{l}}\eta\right)\\
 = & \left(\frac{m\left(1-r^2\right)}{r^2}+\lambda_j\right)^2+\frac{2\left(2\lambda_j-3m^2\right)\left(1-r^2\right)}{r^2}-2m\lambda_j+m^2\\
& +\frac{2}{3}\frac{(m-1)(m-3)}{r^2}\left(\frac{m\left(1-r^2\right)}{r^2}+\lambda_j-m\right), \qquad \mathcal{l}\in \left\{1, 2,\dots, s\right\},
\end{align*}

\begin{align*}
b=b_j = & \left(J_2^c\left(\frac{1}{\sqrt{\lambda_j}}\nabla\alpha_{j,\mathcal{l}}\right),\frac{1}{\sqrt{\lambda_j}}\nabla\alpha_{j,\mathcal{l}}\right)\\
 = & \left(\frac{2-r^2-m}{r^2}+\lambda_j\right)^2+(1-m)(1-m+2\lambda_j)\\
& + \frac{\left(1-r^2\right)\left(4\lambda_j-m^2\right)-\left(2-r^2-m\right)(2m-2)}{r^2}\\
& +\frac{2}{3}\frac{(m-1)(m-3)}{r^2}\left(\frac{2-r^2-m}{r^2}+\lambda_j-m+1\right), \qquad \mathcal{l}\in \left\{1, 2,\dots, s\right\},
\end{align*}
and
\begin{align*}
d=d_j & = \left(J_2^c\left(\alpha_{j,\mathcal{l}}\eta\right),\frac{1}{\sqrt{\lambda_j}}\nabla\alpha_{j,\mathcal{l}}\right)\\
	& = \frac{2\sqrt{1-r^2}}{r}\left(\frac{(m+1)r^2-2}{r^2}-2\lambda_j+3m-1-\frac{2}{3}\frac{(m-1)(m-3)}{r^2}\right)\sqrt{\lambda_j}, \\
	&\mathcal{l}\in \left\{1, 2,\dots, s\right\}.
\end{align*}

From \eqref{matrix} we see that $\left.J_2^c\right|_{S_j}$, $j\geq 1$, has two eigenvalues. Their sum is $a+b$ and their product is $ab-d^2$. 

Now, if we consider the particular case $m=4$ and $r=\sqrt{3}/2$, by a straightforward computation we obtain: for $j=1$, we have $a=b=d=0$, i.e., both eigenvalues of $\left.J_2^c\right|_{S_1}$ are $0$, and for $j\geq 2$ we have $a+b>0$ and $ab-d^2>0$, i.e, both eigenvalues of $\left.J_2^c\right|_{S_j}$ are positive with the same multiplicity. Therefore, we get that the index  of $\left.J_2^c\right|_{S_j}$, $j\geq 1$, is $0$, and the nullity of $\left.J_2^c\right|_{S_j}$ is $2(m+1)=10$ if $j=1$ and $0$ if $j\geq 2$. Next, if we consider $V\in C\left(T\mathbb{S}^m(r)\right)$ with $\Div V=0$ and $\Delta_H V=\mu_k V$,  from \eqref{J2c-divV} we get
$$
J_2^c(V)=\frac{1}{3}\left(\mu_k-8\right)\left(3\mu_k-8\right)V.
$$
Therefore, we get that the index of $\left.J_2^c\right|_{\tilde{S}_k}$ is $0$, for any $k\geq 1$, and the nullity of $\left.J_2^c\right|_{\tilde{S}_k}$ is $m(m+1)/2=10$ if $k=1$ and $0$ if $k\geq 2$.

For the case $m=2$ and $r=1/\sqrt{3}$, i.e., the case $iii)$ in Theorem \ref{thm:inclusion-sphere}, we have:
for $j=1$, we have $a+b>0$ and $ab-d^2=0$, i.e., one eigenvalue of $\left.J_2^c\right|_{S_1}$ is $0$ and the other one is positive, and for $j\geq 2$ we have $a+b>0$ and $ab-d^2>0$, i.e, both eigenvalues of $\left.J_2^c\right|_{S_j}$ are positive. Therefore, we get that the index  of $\left.J_2^c\right|_{S_j}$ is $0$, for any $j\geq 1$,  and the nullity of $\left.J_2^c\right|_{S_j}$ is $m+1=3$ if $j=1$ and $0$ if $j\geq 2$. Further, if we consider $V\in C\left(T\mathbb{S}^m(r)\right)$ with $\Div V=0$ and $\Delta_H V=\mu_k V$,  from \eqref{J2c-divV} we get
$$
J_2^c(V)=\mu_k\left(\mu_k-6\right)V.
$$
Therefore, we get that the index  of $\left.J_2^c\right|_{\tilde{S}_k}$ is $0$, for any $k\geq 1$, and the nullity of $\left.J_2^c\right|_{\tilde{S}_k}$ is $m(m+1)/2=3$ if $k=1$ and $0$ if $k\geq 2$.

We can summarize the above results and state

\begin{Theorem}\label{thm:stability-hypersphere}
We consider $\iota:\mathbb{S}^m(r)\to\mathbb{S}^{m+1}$ the canonical inclusion of the small hypersphere of radius $r$. We have
\begin{itemize}
		\item [i)] if $m=1$ or $m=3$ and $r=1/\sqrt{2}$, then the index of $\mathbb{S}^m(r)$ is $1$ and its nullity is $(m+1)(m+2)/2$;
		\item[ii)] if $m=2$ and $r=1/\sqrt{3}$, then the index of $\mathbb{S}^m(r)$ is $1$ and its nullity is $6$;
		\item[iii)] if $m=4$ and $r=\sqrt{3}/2$, then the index of $\mathbb{S}^m(r)$ is $1$ and its nullity is $20$.
\end{itemize}
\end{Theorem}

\begin{Bem}
We can give a geometric description of the space that provides the index and the kernel of the Jacobi operator $J_2^c$ in the previous theorem. Indeed, when $m=2$, $J_2^c$ is negative definite on
$$
\left\{a\eta \ | \ a\in\mathbb{R}\right\}
$$
and
$$
\Ker J_2^c =  \left\{ V\in C\left(T\mathbb{S}^m(r)\right)\ | \ V \ \text{is Killing}\right\}\oplus \left\{\frac{3\sqrt{2}}{2}\alpha\eta+\nabla\alpha \ |\ \Delta\alpha=\lambda_1\alpha\right\},
$$
where $\lambda_1=6$.

When $m=4$, $J_2^c$ is negative definite on 
$$
\left\{a\eta \ | \ a\in\mathbb{R}\right\}
$$
and
$$
\Ker J_2^c = \left\{ V\in C\left(T\mathbb{S}^m(r)\right)\ | \ V \ \text{is Killing}\right\} \oplus \left\{\nabla\alpha \ |\ \Delta\alpha=\lambda_1 \alpha\right\} \oplus \left\{\alpha\eta \ | \ \Delta\alpha=\lambda_1\alpha\right\},
$$
where $\lambda_1=16/3$.
\end{Bem}

\section{Further results on c-biharmonic stability}


In the following, we present a number of further, more general, statements
on the stability of conformal-biharmonic maps which do not assume that both domain and target manifolds are spheres.

The next result can be considered as a generalization of Theorem \ref{thm:stability-equator} to the case of general target manifolds. Moreover, it can also be seen as a version of Xin's result \cite{MR0587168} for c-biharmonic maps. 

\begin{Theorem}\label{xin-generalization}
Let $\phi:\mathbb{S}^m\to N^n$ be a non-constant harmonic map, considered as a c-biharmonic map. If $m\geq 5$, then $\phi$ is unstable.
\end{Theorem}

\begin{proof}
It is not difficult to check that
$$
J_2^c(V)= J_2(V)+\frac{2(m-3)(m-1)}{3} J(V), \qquad \forall V\in C\left(\phi^{-1}TN\right).
$$
Then, we consider $V=d\phi(\nabla \alpha)$, where $\alpha$ is the restriction of a linear function on $\mathbb{R}^{m+1}$ to $\mathbb{S}^m$. Since $\phi$ is harmonic, we can prove (see \cite{MR0587168}) that
$$
J(V)=(2-m)V,
$$
so
$$
J_2(V)=J(J(V))=(2-m)^2V
$$
and
$$
J_2^c(V)=\frac{(2-m)\left(2m^2-11m+12\right)}{3}V.
$$
As $\phi$ is not a constant map, we can see that there exists $V$ such that it is not the zero section in the pull-back bundle $\phi^{-1}TN$, and the conclusion follows.
\end{proof}

When $\phi$ is a harmonic map from an Einstein domain manifold in a sphere, using an idea of Leung \cite{MR0673586}, we also can give some conditions that force the map to be unstable.

\begin{Theorem}\label{luang-result}
Let $\phi:M^m\to\mathbb{S}^n$ be a non-constant harmonic map, where $M$ is a compact Einstein manifold with $\Ric=\lambda \Id$ and $\lambda$ is a real constant. Consider $\phi$ as a c-biharmonic map and define a covariant $2$-tensor field on $M$, $\omega(X,Y)=\langle d\phi(X),d\phi(Y)\rangle$, for any $X,Y\in C(TM)$. If
\begin{equation}\label{ineq-luang}
\int_M\left(4|\omega|^2+(n-4)|d\phi|^4\right) \ v_g < \frac{2(m-3)(n-2)\lambda}{3}\int_M |d\phi|^2 \ v_g,
\end{equation}
then $\phi$ is unstable.
\end{Theorem}

\begin{proof}
By direct computations we obtain
$$
J_2^c(V)=J_2(V)+\frac{2(m-3)\lambda}{3}J(V),\qquad \forall V\in C\left(\phi^{-1}T\mathbb{S}^n\right).
$$
Then, we consider $V=\left(\nabla^{\mathbb{S}^n} \beta\right)\circ\phi$, where $\beta$ is the restriction of a linear function on $\mathbb{R}^{n+1}$ to $\mathbb{S}^n$. As the map $\phi$ is harmonic, for $\beta\left(\bar{y}\right)=\langle\bar{a}, \bar{y}\rangle$, $\bar{y}\in \mathbb{S}^n$, we obtain 
$$
J(V)=2\tr\langle d\phi(\cdot),\bar{a}\rangle d\phi(\cdot)-|d\phi|^2V
$$  
and 
\begin{align*}
\left(J_2(V),V\right)=\int_M \left(4\left|\tr\langle d\phi(\cdot),\bar{a}\rangle d\phi(\cdot)\right|^2-4|d\phi|^2 \left|\tr \langle d\phi(\cdot),\bar{a}\rangle(\cdot)\right|^2+|d\phi|^4 \left(\left|\bar{a}\right|^2-(\beta\circ\phi)^2\right)\right) \ v_g.
\end{align*}
Next, we consider $\left\{\bar{a}_k\right\}_{k=\overline{1,n+1}}$ to be an orthonormal basis in $\mathbb{R}^{n+1}$ and take $\bar{a}=\bar{a}_k$. Denoting $\beta_k\left(\bar{y}\right)=\langle \bar{a}_k,\bar{y}\rangle$ and $V_k=\left(\nabla^{\mathbb{S}^n}\beta_k\right)\circ \phi$, we obtain
\begin{align*}
	\sum_{k=1}^{n+1}\left(J_2^c\left(V_k\right),V_k\right) = \int_M\left(4|\omega|^2+(n-4)|d\phi|^4 - \frac{2(m-3)(n-2)\lambda}{3} |d\phi|^2 \right)\ v_g.
\end{align*}
Thus, we conclude.
\end{proof}

\begin{Bem}
For an isometric immersion, $|d\phi|^2=m$, $|\omega|^2=m$ and thus \eqref{ineq-luang} becomes
\begin{equation}\label{ineq2-luang}
3\left(4+m(n-4)\right)<2(m-3)(n-2)\lambda.
\end{equation}
When we consider $M^m$ the equator of $\mathbb{S}^{m+1}$ as in Theorem \ref{thm:stability-equator}, since $\lambda=m-1$, from \eqref{ineq2-luang} we reobtain that if $m\geq 5$, then $M$ is unstable. If we consider $M^m$ the minimal generalized Clifford torus as in  Proposition \ref{thm:c-biharmonic-clifford}, i), then $\lambda=m-2$ and we can see that \eqref{ineq2-luang} is satisfied for any $m\geq 6$ and even. Thus, such tori are unstable.
\end{Bem}

Since for any map $|\omega|^2\leq |d\phi|^4$, in the hypotheses of Theorem \ref{luang-result}, we also can state the following.

\begin{Cor}
If
$$
n\int_M|d\phi|^4 \ v_g < \frac{2(m-3)(n-2)\lambda}{3}\int_M |d\phi|^2 \ v_g,
$$
then $\phi$ is unstable.
\end{Cor}

Further, we note that Theorem \ref{thm:stability-hypersphere} can be seen as a particular case of the following result which involves the stress-energy tensor $S$ associated with the energy functional $E$.

\begin{Theorem}
Let $\phi:M^m\to\mathbb{S}^n$ be a non-harmonic c-biharmonic map, where $M$ is a compact Einstein manifold with $\Ric=\lambda \Id$ and $\lambda$ is a real constant. If $\Div S=0$ or $|\tau(\phi)|$ is constant, then $\phi$ is unstable.
\end{Theorem}

\begin{proof}
We know that
$$
J_2^c(V)=J_2(V)+\frac{2(m-3)\lambda}{3}J(V),\qquad \forall V\in C\left(\phi^{-1}T\mathbb{S}^n\right).
$$
We choose $V=\tau(\phi)$ and, denoting $\theta(X)=\langle d\phi(X), \tau(\phi)\rangle=-(\Div S)(X)$, for any $X\in C(TM)$, we infer
$$
J(\tau(\phi))=-\frac{2(m-3)\lambda}{3}\tau(\phi), 
$$
$$
J_2(\tau(\phi))=\left(\frac{4(m-3)^2\lambda^2}{9}+4\Div\theta^{\#}-4|\tau(\phi)|^2\right)\tau(\phi)-2d\phi\left(\nabla \left(|\tau(\phi)|^2\right)\right),
$$
where $\theta^{\#}$ is the vector field associated to the $1$-form $\theta$ by the musical isomorphism.
Thus, since $\theta=0$ of $|\tau(\phi)|$ is constant we obtain
$$
\left(J_2^c(\tau(\phi)),\tau(\phi)\right)=-4\int_M|\tau(\phi)|^4\ v_g<0.
$$
\end{proof}

From the proof of Theorem \ref{thm:stability-hypersphere} we see that if $V\in C\left(T\mathbb{S}^m(r)\right)$, then $\left(J_2^c(V), V\right)\geq 0$. In a different way, without using the splitting technique, we can prove the following result concerning the contribution of the tangent vector fields.

\begin{Prop}
Let $\iota:\mathbb{S}^m(r)\to\mathbb{S}^{m+1}$ be the c-biharmonic canonical inclusion given in Theorem \ref{thm:inclusion-sphere}, $r\in (0,1)$, and consider $V\in C\left(T\mathbb{S}^m(r)\right)$. Then, $\left(J_2^c(V), V\right)\geq 0$ and $J_2^c(V)=0$ if and only if
\begin{itemize}
\item [i)] $m\in\{1,2,3\}$ and $V$ is a Killing vector field;
\item[ii)] $m=4$ and either $V$ is a Killing vector field, or $V=\nabla \alpha$ with $\Delta\alpha=\lambda_1\alpha$.
\end{itemize}
\end{Prop}

\begin{proof}
By a straightforward computation we have
$$
J(V)=\bar{\Delta}V-(m-1)V,
$$
and
$$
J_2(V)=\bar{\Delta}\bar{\Delta}V-(2m-1)\bar{\Delta}V+\left(\bar{\Delta}V\right)^\top+\left((m-1)^2-\frac{m^2\left(1-r^2\right)}{r^2}\right)V-\frac{2m\sqrt{1-r^2}}{r}\Div V\eta,
$$
where $(\cdot)^\top$ denotes the tangent part of a section in the pull-back bundle.
Now, since 
$$
J_2^c(V)=J_2(V)+\frac{2}{3}\frac{(m-1)(m-3)}{r^2} J(V),
$$
we achieve 
\begin{align*}
\left(J_2^c(V),V\right)=  \int_M & \left(\left|\bar{\Delta}V-(m-1)V\right|^2-\frac{1}{r^2}\left(m^2\left(1-r^2\right)+\frac{2}{3}(m-1)^2(m-3)\right)|V|^2\right.\\
&\quad \left. +\frac{2}{3}\frac{(m-1)(m-3)}{r^2}\langle \bar{\Delta}V,V\rangle\right)\ v_g.
\end{align*}
Further, it is not difficult to see that
$$
\bar{\Delta}V=\frac{2\sqrt{1-r^2}}{r}\Div V\eta +\Delta_H V-2\Ric V+\frac{m-r^2}{r^2}V.
$$
As the hypersphere $\mathbb{S}^m(r)$ is c-biharmonic in $\mathbb{S}^{m+1}$, we know that 
$$
r^2=\frac{2m^2-5m+6}{6m}.
$$
Using also the integral formula of Yano that holds for any tangent vector field (see \cite{MR0046122})
$$
\int_M\langle \Delta_HV-2\Ric V,V\rangle\ v_g=\int_M\left(\frac{1}{2}\left|\mathcal{L}_Vg\right|^2 - (\Div V)^2\right)\ v_g, 
$$
where $\mathcal{L}_Vg$ denotes the Lie derivative of the domain metric $g$ with respect to $V$, we obtain
$$
\left(J_2^c(V),V\right)= \int_M \left(\left|\Delta_HV-2\Ric V\right|^2+\frac{1}{r^2}\left(-\frac{7m^2-22m+12}{3m}(\Div V)^2+\frac{m}{2}\left|\mathcal{L}_Vg\right|^2\right)\right)\ v_g.
$$
Next, taking into account that 
$$
\left|\mathcal{L}_Vg\right|^2\geq \frac{4}{m}(\Div V)^2,
$$
see for example \cite{MR2135286}, if follows that
$$
\left(J_2^c(V),V\right)\geq \int_M\left(\left|\Delta_HV-2\Ric V\right|^2-\frac{1}{r^2}\frac{7m^2-28m+12}{3m}(\Div V)^2\right) \ v_g.
$$
Now, if $\Div V=0$, then $\left(J_2^c(V),V\right)\geq 0$ and we can see that $\left(J_2^c(V),V\right)=0$ implies $V$ is a Killing vector field and, in fact, $J_2^c(V)=0$. 

If $V=\nabla \alpha$ and $m=1$ or $m=3$, we get $\left(J_2^c(V),V\right)>0$.

If $V=\nabla \alpha$ with $\Delta\alpha=\lambda_j\alpha$, $j\geq 1$ and $m=2$ or $m=4$, we obtain
\begin{equation*}
\left(J_2^c(V),V\right)\geq\lambda_j\left(\lambda_j^2-\frac{1}{r^2}\frac{19m^2-40m+12}{3m}\lambda_j+\frac{4(m-1)^2}{r^4}\right)\int _M \alpha^2\ v_g.
\end{equation*} 
Now, by standard arguments, we conclude.
\end{proof}

\bibliographystyle{plain}
\bibliography{mybib}

\end{document}